\documentclass[11pt]{amsart}
\usepackage{geometry}                
\usepackage{graphicx}
\usepackage{amssymb}
\usepackage{epstopdf}
\usepackage{lipsum}
\usepackage{color}
\usepackage{hyperref}
\newtheorem{theorem}{Theorem}
\newtheorem{prop}[theorem]{Proposition}
\newtheorem{cor}[theorem]{Corollary}
\newtheorem{lemma}[theorem]{Lemma}

\newtheorem{theore}{Theorem}
\newtheorem{theor}{Theorem}
\newtheorem{theo}{Theorem}
\theoremstyle{definition}
\newtheorem{definition}[theorem]{Definition}
\newtheorem{example}[theorem]{Example}
\newtheorem{remark}[theorem]{Remark}

\newcommand{\idiot}[1]{\vspace{5 mm}\par \noindent
\marginpar{\textsc{Note}}
\framebox{\begin{minipage}[c]{0.95 \textwidth}
#1 \end{minipage}}\vspace{5 mm}\par}

\renewcommand{\idiot}[1]{}

\def\la{{\lambda}}
\def\al{{\alpha}}
\def\be{{\beta}}
\def\ga{{\gamma}}

\def\QQ{{\mathbb Q}}

\def\CC{{\mathbb C}}

\def\st{{\tilde s}}

\def\mub{{\bar \mu}}
\def\hht{{\tilde h}}

\def\mt{{\tilde m}}
\def\bP{{\mathcal P}}
\def\bT{{\mathcal T}}
\def\bC{{\mathcal C}}

\def\bfp{{\mathbf p}}
\def\bfpb{{\overline \bfp}}

\def\mvdash{{\,\vdash\!\!\vdash}}

\def\HX{{H}}

\def\gb{{\overline g}}

\def\mub{{\overline \mu}}
\def\gab{{\overline \ga}}

\def\dcl{{\{\!\!\{}}
\def\dcr{{\}\!\!\}}}

\newcommand{\pchoose}[2]{\begin{pmatrix}#1\\ #2\end{pmatrix}}

\newdimen\squaresize \squaresize=10pt
\newdimen\thickness \thickness=0.4pt

\def\square#1{\hbox{\vrule width \thickness
     \vbox to \squaresize{\hrule height \thickness\vss
        \hbox to \squaresize{\hss#1\hss}
     \vss\hrule height\thickness}
\unskip\vrule width \thickness}
\kern-\thickness}

\def\vsquare#1{\vbox{\square{$#1$}}\kern-\thickness}

\def\young#1{
\vbox{\smallskip\offinterlineskip
\halign{&\vsquare{##}\cr #1}}}

\def\thisbox#1{\kern-.09ex\fbox{#1}}
\def\downbox#1{\lower1.200em\hbox{#1}}
\newdimen\Squaresize \Squaresize=15pt
\newdimen\Thickness \Thickness=0.4pt

\def\Square#1{\hbox{\vrule width \Thickness
     \vbox to \Squaresize{\hrule height \Thickness\vss
        \hbox to \Squaresize{\hss#1\hss}
     \vss\hrule height\Thickness}
\unskip\vrule width \Thickness}
\kern-\Thickness}

\def\Vsquare#1{\vbox{\Square{$#1$}}\kern-\Thickness}

\def\Young#1{
\vbox{\smallskip\offinterlineskip
\halign{&\Vsquare{##}\cr #1}}}

\title[Character symmetric functions]{Symmetric group characters
as symmetric functions}
\author[Rosa Orellana]{Rosa Orellana}%
\address{Dartmouth College, Mathematics Department, Hanover, NH 03755, USA} \email{rosa.c.orellana@dartmouth.edu}%

\author[Mike Zabrocki]{Mike Zabrocki}%
\address{Department of Mathematics and Statistics, York University, Toronto, Ontario M3J 1P3,
Canada} \email{zabrocki@mathstat.yorku.ca}%

\date{}  

\thanks{Work supported by NSF grants DMS-1300512 and DMS-1458061, and by NSERC}

\begin{document}

\begin{abstract}
We introduce a basis of the symmetric functions that evaluates to the (irreducible) characters of the symmetric group, just 
as the Schur functions evaluate to the irreducible characters of $GL_n$ modules.
Our main result gives three different characterizations for this basis.
One of the characterizations shows that the structure coefficients for the
(outer) product of these functions are the stable Kronecker coefficients.
The results in this paper focus on developing the fundamental properties of this basis.
\end{abstract}

\maketitle

\tableofcontents

\begin{section}{Introduction}

The ring of symmetric functions and the representation theory of both the symmetric group,
$S_k$, and the general linear group, $GL_n$, are deeply connected.  It is well-known that
the ring of symmetric functions $Sym$ is the ``universal'' character ring of $GL_n$. In fact,
irreducible polynomial $GL_n$ characters are obtained as evaluations of Schur functions, $s_\lambda$,
at the eigenvalues of the elements in $GL_n$.  For this reason, Schur functions serve as a tool to
decompose $GL_n$-modules in term of irreducible submodules.  The main objective of this paper is to show
that the embedding of the symmetric group, $S_n$, as permutation matrices in $GL_n$ gives rise to a
basis $\{ \st_\la\}$ of the ring of symmetric functions which evaluate to the characters of $S_n$
in a similar manner.

Schur functions are also useful in computing the decomposition of the tensor product 
of polynomial irreducible representations of $GL_n$. It is well known that this tensor product corresponds
to the (outer) product of Schur functions. More specifically,
for any three partitions $\la, \mu$ and $\nu$ such that $|\nu| = |\la|+|\mu|$, the
multiplicity of irreducible representation $W^\nu$ in $W^\lambda \otimes W^\mu$ are the Littlewood-Richardson
coefficients  (denoted $c_{\lambda\mu}^\nu$).  These are the same integers that occur as the structure 
coefficients for the Schur basis, that is,
\[s_\la s_\mu = \sum_\nu c_{\la\mu}^\nu s_\nu . \]

The multiplicities of the tensor product of irreducible modules of the symmetric group have some dependence on $n$,
but for partitions $\la$, $\mu$ and $\nu$ and $n$ sufficiently large,
the multiplicity of the irreducible representation $S^{(n-|\nu|,\nu)}$ in
$S^{(n-|\la|,\la)} \otimes S^{(n-|\mu|,\mu)}$ is independent of $n$ and the values are known
as the stable or reduced Kronecker coefficients (denoted $\overline{g}_{\la\mu}^\nu$,
see for instance \cite{BOR1, Gu, Murg2, Murg3}).
One of our main results (Theorem \ref{th:uniqueness} part (3)) is a characterization of 
symmetric functions $\st_\lambda$ as a family of elements whose structure coefficients are the
reduced Kronecker coefficients, that is,
\[\st_\la \st_\mu = \sum_\nu \overline{g}_{\la\mu}^\nu \st_\nu . \]

Our main theorem is the following characterizations of the basis $\{\st_\lambda\}$ of
inhomogeneous degree:

\begin{theorem}\label{th:uniqueness}
There is a unique (in-homogeneous) basis of the symmetric functions, $\{\st_\lambda\}$, that are
characterized by any of the following three properties:
\begin{enumerate}
\item For a fixed partition $\lambda$, $\st_\la$ is
the unique symmetric function with the property that
for all $n \geq |\la|+\la_1$ and for all partitions $\gamma$ of $n$,
$$\st_\lambda(x_1, x_2, \ldots, x_n) = \chi^{(n-|\la|,\la)}(\ga),$$
where $x_1, x_2, \ldots, x_n$ are the eigenvalues of a permutation matrix of cycle structure 
$\ga$ and $\chi^{(n-|\la|,\la)}(\ga)$
are the values of the irreducible characters of the symmetric group. 
\item The set $\{ \st_\lambda \}$ is
the unique family of symmetric functions such
that, for a sufficiently large $n$, if the multiplicity of
the $S_n$ module $S^{(n-|\mu|,\mu)}$ in the decomposition of the restriction
of the $GL_n$ module $W^\lambda$ to $S_n$ is $r_{\la\mu}$,
that is,
$$W^\lambda\downarrow_{S_n}^{GL_n}
\simeq \bigoplus_{\mu}
(S^{(n-|\mu|,\mu)})^{\oplus r_{\la\mu}},
\hbox{ then }
s_\lambda = \sum_{\mu} r_{\la\mu} \st_{\mu}~.$$
\item The set $\{ \st_\lambda \}$ is
the unique family of symmetric functions such
that $s_{1^r} = \st_{1^r} + \st_{1^{r-1}}$ and
$$\st_\la \st_\mu = \sum_{\nu}
{\overline g}_{\la\mu}^\nu \st_\nu~.$$
\end{enumerate}
\end{theorem}

We call this new basis the
{\it irreducible character basis} (or {\it $\st$-basis}). The basis $\{{\tilde s}_\lambda\}$ 
plays the same role for the symmetric group as the Schur functions
$\{ s_\lambda\}$ do for $GL_n$.  

If we use the notation $G(GL_n)$ (respectively, $G(S_n)$)
to denote the Grothendieck ring of  (polynomial) $GL_n$-representations (respectively, $S_n$),
then the map $ch: G(GL_n)\rightarrow Sym$, can be defined via the trace function
(see, for instance, \cite{Stanley} Prop. A2.3 for details). Our main theorem
implies that there exists a map $\tilde{ch}$ which maps the irreducible
representation, $S^{(n-|\lambda|,\lambda)}$ of $S_n$ to $\st_\lambda\in Sym$ such
that the following diagram commutes.

\begin{center}
\includegraphics[width=1.75in]{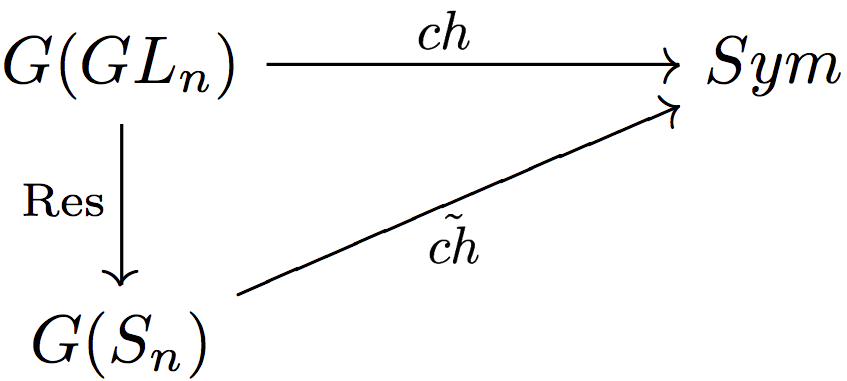}~.
\end{center}
The $\tilde{ch}$ map can be defined as a lift of the trace map on the representations of $S_n$.

The expansion of a Schur function in terms of the $\st$-basis, {\it i.e}, $s_\lambda = \sum_\mu r_{\la\mu} \st_\mu$, is equivalent to what we call the ``restriction problem" which asks for the decomposition of a (polynomial) irreducible representation of $GL_n$, $W^\lambda$ when restricted to the symmetric group.  This problem has been studied in the literature,  \cite{ButlerKing, King, Lit, Nis, ScharfThibon, STW, H, HSW, NPPS1, NPPS2}.  The coefficient $r_{\la\mu}$ is the multiplicity of the irreducible $S_n$ module indexed by $(n-|\mu|,\mu)$ in the restriction of the irreducible representation $W^\lambda$ of $GL_n$.  In \cite{Lit}, Littlewood gives a formula for this using plethysms and Scharf
and Thibon  \cite{ScharfThibon} give a proof using Hopf algebra techniques that the coefficients have the expression
$$r_{\la\mu} = \left< s_\la, s_{(n-|\mu|,\mu)}[1+s_1+s_2+\cdots] \right>$$
for an integer $n$ which is sufficiently large.
Details for computing the plethysm  $f[1+s_1+s_2+\cdots]$ can be found in  \cite{Mac, Stanley}.  Finding an explicit 
combinatorial expression for the coefficients $r_{\la\mu}$ remains a motivating open problem.

To study the restriction problem, Butler and King \cite{ButlerKing} considered the embedding of $S_n$ in $GL_{n-1}$. 
In their paper they describe a method for writing symmetric group characters in terms of Schur functions.   
However, they fell short of defining symmetric group characters as a basis of symmetric functions.  
The embedding they considered would lead to symmetric functions that differ from ours by a plethystic substitution. 
Using our notation, Table II of \cite{ButlerKing} is the Schur expansion of the symmetric functions $\st_\lambda[X+1]$ 
for partitions $|\la| \leq 5$.

The irreducible characters of $S_k$, $\chi^\lambda$, are well known to occur as change of basis
coefficients when we write the power symmetric functions, $p_\mu$, in terms of the Schur functions,
\begin{equation}\label{ptos}
p_{\mu} = \sum_{\lambda \vdash k}\chi^\lambda(\mu) s_\lambda
\end{equation}
where $\chi^\lambda(\mu)$ is the value of $\chi^\lambda$ at the conjugacy class indexed by $\mu$. 

In the early 1990s Martin \cite{Ma1,Ma2,Ma3,Ma4} and Jones \cite{Jo} introduced the
{\it partition algebra}, $P_k(n)$. Jones showed that if we
identify $S_n$ with the permutation matrices in $GL_n$ and
act diagonally on $V^{\otimes k}$, then the centralizer
algebra we obtain is isomorphic to $P_k(n)$ when $n\geq 2k$.
This means that $S_n$ and $P_k(n)$ are Schur-Weyl duals of each other.

The development of the representation theory of the partition algebra
\cite{BDE, BDO, BH, BHH, Hal, HalRam, Jo, Ma1, Ma2, Ma3, Ma4}
has shown that there are close connections between the characters of the partition algebra 
and the characters of the symmetric group considered through the embedding using permutation matrices.  
In fact, we find that the characters of the partition
algebra appear as the change of basis coefficients
between the power sum basis and the irreducible character
basis. That is, in analogy to  (\ref{ptos}) we have
\[p_{\mu} = \sum_{\lambda}\chi^\lambda_{P_k(n)}(d_\mu) \st_\lambda\]
where $\chi^\lambda_{P_k(n)}(d_\mu)$ denotes the irreducible
character of $P_k(n)$ indexed by $\lambda$ and evaluated at an
element $d_\mu$ which is analogous to a conjugacy class representative
(see \cite{Hal} for further details).

We remark that our new basis provides a unifying mantle for many different objects connected with
Kronecker coefficients such as the representation theory of the symmetric group and that of the
partition algebra, character polynomials, etc.  In particular, in Section \ref{sec:MNrule},
we provide a symmetric function version of the Murnaghan-Nakayama rule for computing the
characters of the partition algebra, these were computed in \cite{Hal} via different methods.
We also show in Section  \ref{sec:stcharpoly}  that character polynomials can be obtained as
evaluations of our symmetric functions using results in \cite{GG}.   Further, the study of our
new basis has led to the introduction of new combinatorial objects that will be useful  in
developing algorithms involving symmetric group characters and partition algebra characters.

The reader interested in computing $\st_\lambda$ from the formulae in this
paper is recommended to use Equation \eqref{eq:pexpansionofst}
because this expression provides an explicit expansion in terms of the power sum basis.
The definition (see \eqref{eq:definingrel} and \eqref{eq:sttoht})
makes a combinatorial connection with multiset tableaux and can be used to
compute the elements recursively.

In this paper, we also introduce another basis of symmetric functions $\{\hht_\lambda\}$
that evaluate to the induced trivial characters from a Young subgroup of the symmetric group 
to the full symmetric group.  This basis has a close combinatorial connection with multiset partitions 
and provides a useful tool in developing the change of basis coefficients.

In the third characterization of Theorem \ref{th:uniqueness}, we saw that the structure coefficients
of the $\st$-basis are the stable Kronecker coefficients. Therefore, our new basis connects to the Grothendieck
ring of the tensor category $Rep(S_t)$ of Deligne \cite{Deligne, EA}, which has simple objects that have
as structure coefficients the stable Kronecker coefficients.  In addition, our basis should provide a
more compact way to encode the characters of FI-modules studied by Church, Ellenberg and Farb
\cite{ChurchFarb, ChurchEllenbergFarb}.

Similar results have been obtained for other subgroups of $GL_n$. 
For example, Koike and Terada \cite{KT} developed symmetric function
expressions for characters of orthogonal and symplectic groups.

Additional information about this basis can be found in two followup publications \cite{OZ1, OZ2}.

\subsection*{Acknowledgements} The authors used
Sage \cite{sage, sage-co} to perform computations and benefited
from the work of the developers of this platform.
After we posted this paper, Sami Assaf and David Speyer
shared with us their research notes where they
had independently discovered the same symmetric functions.  Their results,
which identify the sign of a coefficient of a Schur function
in this basis, appear in \cite{AssafSpeyer}.

\end{section}

\begin{section}{Notation and Preliminaries}

In this section we provide definitions of the combinatorial objects arising in this work and establish notation and conventions.

For non-negative integers $n$ and $\ell$, a {\it partition}
of size $n$ and length $\ell$ is a sequence of positive integers,
$\lambda = (\lambda_1, \lambda_2, \ldots, \lambda_{\ell})$,
such that $\lambda_1 \geq \lambda_2 \geq  \cdots \geq \la_{\ell} > 0$
and $\lambda_1 + \lambda_2 + \cdots + \lambda_\ell = n$.
The {\it size} of the partition is denoted $|\la| = n$ and
the {\it length} of the partition is denoted $\ell(\la)=\ell$.
We will often use the shorthand notation $\la \vdash n$
to indicate that $\lambda$ is a partition of $n$.
The symbols $\lambda$ and $\mu$ will be reserved exclusively
for partitions.  Let $m_i(\la)$ represent the number of
times that $i$ appears in the partition $\la$.  Sometimes it will be
convenient to represent our partitions in
exponential notation $\la = (1^{m_1}2^{m_2}\cdots k^{m_k})$, where  
$m_i = m_i(\la)$ is the number of times that part $i$ occurs in $\lambda$.
Using this notation
the number of permutations with cycle structure $\lambda\vdash n$ is
$\frac{n!}{z_\la}$ where
\begin{equation}
z_\la = \prod_{i=1}^{\la_1} m_i(\la)! i^{m_i(\la)}~.
\end{equation}
The most common operation we use is that of adding a
part of size $n$ to the beginning of a partition.  This
is denoted $(n, \la)$.  If $n < \la_1$, this
sequence will no longer be an integer partition and we will have
to interpret the object appropriately.

The {\it Young diagram} of a partition
$\la$ are the set of points (or cells)
$\{ (i,j) : 1 \leq i \leq \la_j, 1 \leq j \leq \ell(\la) \}$.
We will represent these cells as stacks of boxes in
the first quadrant (following `French notation').  
A {\it tableau} is a map from the set of
cells in the diagram of the partition to a set of labels.  
We represent a tableau by filling the cells of the 
diagram of a partition with the labels. In our case, we will
encounter tableaux where only a subset of the cells are
mapped to labels (some boxes will be empty).  A tableau $T$ is {\it column strict} if
$T(i,j)\leq T(i+1,j)$ and $T(i,j)<T(i,j+1)$ for all the
filled cells of the tableau.  The {\it content} of a tableau is the 
multiset containing the total number of occurrences of each number.

A {\it multiset} is a set where elements can be repeated.  
To differentiate multisets from sets we use double brackets to denote multisets, i.e.,  
$\dcl b_1, b_2, \ldots, b_r \dcr$.  Multisets will also be represented by
exponential notation, in this case the multiset 
$\dcl 1^{a_1},2^{a_2},\ldots, \ell^{a_\ell}\dcr$
represents the multiset where the element $i$ occurs
$a_i$ times.

A {\it set partition} of a set $S$ is a set of pairwise disjoint subsets
$\{ S_1, S_2, \ldots, S_\ell\}$ such that $S_i \subseteq S$ for
$1 \leq i \leq \ell$ and $S_1 \cup S_2 \cup \cdots \cup S_\ell = S$.
A {\it multiset partition} $\pi = \dcl S_1, S_2, \ldots, S_\ell \dcr$
of a multiset $S$ is a similar construction to a
set partition, but now each $S_i$ may be
a multiset, and it is possible that two multisets
$S_i$ and $S_j$ (with $i\neq j$) have non-empty intersection (and may
even be equal).  The {\it length} of a multiset partition is denoted by
$\ell(\pi) = \ell$.  We will use the notation
$\pi \mvdash S$ to indicate that $\pi$ is a multiset
partition of the multiset $S$.

We will use $\mt(\pi)$ to represent the partition of the integer $\ell(\pi)$
consisting of the multiplicities of the multisets which
occur in $\pi$. For example,  $\mt(\dcl\dcl 1,1,2\dcr,\dcl 1,1,2\dcr,\dcl 1,3\dcr\dcr)
=(2,1)$ because $\dcl 1,1,2\dcr$ occurs 2 times and
$\dcl 1,3\dcr$ occurs 1 time.

For non-negative integers $n$ and $\ell$, a {\it composition} of size $n$
is an ordered sequence  of positive integers $\alpha = (\alpha_1, \al_2, \ldots, \al_\ell)$
such that
$\al_1+\al_2 + \cdots +\al_\ell=n$.  A {\it weak composition}
is such a sequence with the condition that $\al_i \geq 0$ (zeros are allowed).
To indicate that $\alpha$ is a composition of $n$ we will
use the notation
$\alpha \models n$ and to indicate that
$\alpha$ is a weak composition of $n$ we will use the notation
$\alpha \models_w n$.  For both
compositions and weak compositions, $\ell(\alpha) := \ell$ (the length of the sequence).

\begin{remark} Multiset partitions of a multiset are isomorphic to
objects called vector partitions
which have previously been used to
index symmetric functions in multiple sets of variables \cite{Comtet, MacMahon, Rosas}.
Since multiset partitions are
more amenable to tableaux we have used this alternate combinatorial
description.
\end{remark}

\begin{subsection}{The ring of symmetric functions.}
\label{subsec:ringsf}
For complete details on this topic 
we refer the reader to  \cite{Mac, Sagan, Stanley, Lascoux}.
The ring of symmetric functions will be denoted
$Sym = \QQ[p_1, p_2, p_3, \ldots]$.  The $p_k$ are
power sum generators and they will be thought of as
functions which can be evaluated at values when
appropriate by making the substitution
$p_k \rightarrow x_1^k + x_2^k + \ldots + x_n^k$
but they are used algebraically in this ring
without reference to their variables.

The fundamental bases of $Sym$ used in this paper (each indexed by the set of partitions $\lambda$) are
{\it power sum} $\{p_\la\}_{\la}$,
{\it homogeneous/complete} $\{h_\la\}_{\la}$,
{\it elementary} $\{e_\la\}_{\la}$, and
{\it Schur} $\{s_\la\}_{\la}$.
For consistent notation, $p_0 = h_0 = e_0 = 1$ in this ring.
The Hall inner product is defined by declaring that the
power sum basis is orthogonal, i.e.,
$$\left< p_\la, p_\mu \right>
= z_\la \delta_{\la=\mu},$$ where
the symbol $\delta_{\la=\mu}$ is the Kronecker delta function
that is equal to $1$ if $\la=\mu$ and $0$ otherwise.
Under this inner product the Schur functions are orthonormal,
$\left<s_\lambda,s_\mu \right>=\delta_{\lambda=\mu}$. We use this
scalar product to represent values of coefficients by taking scalar products with dual bases.
We will also refer to the irreducible character of the symmetric group indexed by the partition
$\la$ and evaluated at a permutation
of cycle structure $\mu$ as the coefficient
$\left< s_\la, p_\mu \right> = \chi^\la(\mu)$.

For $k > 0$, define
$$\Xi_k := 1, e^{2\pi i/k}, e^{4\pi i/k},
\ldots, e^{2(k-1)\pi i/k}$$
denote the set of eigenvalues of a permutation
matrix of a $k$-cycle.  Then for any partition
$\mu$, let
$$\Xi_\mu := \Xi_{\mu_1}, \Xi_{\mu_2},
\ldots, \Xi_{\mu_{\ell(\mu)}}$$
be the multiset of eigenvalues of a permutation matrix with
cycle structure $\mu$.  We will evaluate symmetric
functions at these eigenvalues.  The notation
$f[\Xi_\mu]$ represents the complex number we obtain when we take
$f \in Sym$ and replace $p_k$ in $f$ with
$x_1^k + x_2^k + \cdots + x_{|\mu|}^k$ and then
replacing the variables $x_i$ with the values in $\Xi_\mu$.
\end{subsection}
\end{section}

\begin{section}{Symmetric group character bases
of the symmetric functions}\label{sec:symgrpchar}

In this section we introduce two new (in-homogeneous) bases of the ring of symmetric functions,
$\{\hht_\la\}$ and $\{\st_\la\}$.  The evaluations of these
families of symmetric functions at roots of unity (the eigenvalues of a
permutation matrix) will be the values of characters of the symmetric group.
We have relegated a large part of the necessary buildup
of these symmetric function bases to two appendices in Section
\ref{sec:specializations} and \ref{sec:rootsofunity}.  The reason for this is that it will make more
clear the goals of this paper which are the introduction of these symmetric functions and the study of
their properties.   In Section \ref{sec:specializations}, we prove the following fundamental result.

\begin{theorem}\label{thm:hrts}
For all partitions $\nu$ and $\mu$, let $\HX_{\nu,\mu} :=
\left< h_{|\mu|-|\nu|} h_\nu, p_\mu \right>$.  We have
the evaluation,
\begin{equation}\label{eq:hlainhtbasis}
h_\lambda[\Xi_\mu] =
\sum_{\pi \mvdash \dcl1^{\la_1},2^{\la_2},\ldots,\ell^{\la_\ell}\dcr}
\HX_{\mt(\pi),\mu}~.
\end{equation}
\end{theorem}
\vskip .1in
Note that $\HX_{\nu,\mu} = 0$ if $|\mu|-|\nu|<0$.

\begin{definition}
We take as definition of symmetric function elements $\hht_\nu$
the equation
\begin{equation}\label{eq:definingrel}
h_\lambda =
\sum_{\pi \mvdash \dcl1^{\la_1},2^{\la_2},\ldots,\ell^{\la_\ell}\dcr}
\hht_{\mt(\pi)}.
\end{equation}
\end{definition}
This is a recursive definition for calculating this basis
since there is precisely one multiset partition
of $\dcl1^{\la_1},2^{\la_2},\ldots,\ell^{\la_\ell}\dcr$
such that $\mt(\pi)$ is of size $|\la|$, hence
\begin{equation}
\hht_\la = h_\la - \sum_{\substack{\pi \mvdash \dcl1^{\la_1},2^{\la_2},\ldots,\ell^{\la_\ell}\dcr\\\mt(\pi)\neq\la}}
\hht_{\mt(\pi)}~.
\end{equation}

Now by equation \eqref{eq:hlainhtbasis} and an induction
argument, we conclude that for all partitions $\mu$,
\begin{equation}
\hht_{\la}[\Xi_\mu] = \HX_{\la,\mu} = \left< h_{|\mu|-|\la|}h_{\la}, p_\mu\right>
\end{equation}
and this is the value of the character of the
trivial representation induced from
$S_{|\mu|-|\la|} \times S_{\la_1} \times S_{\la_2} \times
\cdots \times S_{\la_{\ell(\la)}}$ to the full symmetric group
$S_{|\mu|}$, i.e., $\mathbf{1}\uparrow_{S_{|\mu|-|\la|} \times S_{\la_1} \times S_{\la_2} \times
\cdots \times S_{\la_{\ell(\la)}}}^{S_{|\mu}}$.  We call these symmetric functions `characters'
because we can think of them as functions that can be
evaluated on the eigenvalues of permutation matrices and these
evaluations are equal to the character values of symmetric
group representations.  Therefore, we
call the symmetric functions $\{ \hht_\lambda\}$, as a
basis of the symmetric functions, the {\it induced trivial
character basis} (or $\hht$-basis).

\begin{example}
To compute a small example: $\hht_1 = h_1$.
Since $\dcl\dcl1\dcr,\dcl1\dcr\dcr$ and
$\dcl\dcl1,1\dcr\dcr$ are the two multiset partitions of $\dcl1,1\dcr$,
then $h_2 = \hht_1 + \hht_2$. Therefore $\hht_2 = h_2 - h_1$.
\end{example}

It is well known that two polynomials of degree
$d$ that agree on $d+1$ values are equal.
We use a multivariate formulation of this idea
repeatedly in order to justify some of our
symmetric function identities.  The details of this
proof are left to Appendix II
(Section \ref{sec:rootsofunity}).

\begin{prop} (Proposition \ref{prop:rootsimplsf} and
Corollary \ref{cor:evalsf})
Let $f,g \in Sym$ be symmetric functions
of degree less than or equal to some positive integer $n$.
Assume that $$f[\Xi_\ga] = g[\Xi_\ga]$$ for all partitions
$\ga$ such that $|\ga| \leq n$
(respectively, $|\ga|\geq n$), then
$$f=g$$
as elements of $Sym$.
\end{prop}

Next we define the symmetric functions $\st_\la$ by using the
Kostka coefficients (the change of basis coefficients between
the induced trivial characters $\mathbf{1}\uparrow_{S_{\mu_1}\times \cdots \times S_{\mu_\ell}}^{S_n}$
and the irreducible characters $\chi^\la$ are denoted $K_{\la\mu}$)
as the change of basis coefficients with $\hht_\la$ basis.
Choose an $n \geq 2|\mu|$.
We define $\st_\lambda$ to be the unique symmetric function which satisfy
\begin{equation}\label{eq:httost}
\hht_\mu = \sum_{|\la|\leq|\mu|} K_{(n-|\la|,\la)(n-|\mu|,\mu)} \st_\la~.
\end{equation}
Alternatively, the coefficient of $\st_\la$ in $\hht_\mu$
is equal to $\sum_{\ga} K_{\ga\mu}$, where the sum is over
partitions $\ga$ such that $\ga/\lambda$ is a horizontal strip
(at most one cell in each row) of size $|\mu|-|\la|$.
This also implies that we can express $\st_\la$ in
terms of the $\hht$'s as
\begin{equation}\label{eq:sttoht}
\st_\la = \sum_{|\mu|\leq|\la|}
K_{(n-|\la|,\la)(n-|\mu|,\mu)}^{-1} \hht_\mu ~, 
\end{equation}
where $n$ is any positive integer greater than
or equal to $2|\la|$ and
$K_{\la\mu}^{-1}$ are the inverse Kostka coefficients.
There is a combinatorial interpretation
for the Kostka coefficients $K_{\lambda\mu}$
as the number of column
strict tableaux of shape $\lambda$ and content $\mu$.
Using this interpretation we can show that $K_{(n-|\la|,\la)(n-|\mu|,\mu)}$ is independent of the value of $n$ as long as $n$ is sufficiently large. That is, if $n$ is sufficiently large, then for any $m\geq n$, $K_{(m-|\la|,\la)(m-|\mu|,\mu)}=K_{(n-|\la|,\la)(n-|\mu|,\mu)}$.

If $n$ is smaller than $|\la|-\la_1$, then the
change of basis coefficients are the same as those
between the complete symmetric functions and
a Schur function indexed by a composition
$\alpha = (|\mu|-|\la|,\la)$, namely
the expression representing the Jacobi-Trudi
matrix
\begin{equation}
s_\alpha = \det\left[ h_{\alpha_i + i - j}
\right]_{1 \leq i,j \leq \ell(\la)+1}~.
\end{equation}

Thus, it follows that the $\st_\la$ are the (unique) in-homogeneous symmetric
functions of degree $|\la|$ that evaluate to the irreducible characters of the symmetric group, that is
\begin{align}
\st_\la[\Xi_\ga]
&= \sum_{|\mu|\leq|\la|}
K_{(n-|\la|,\la)(n-|\mu|,\mu)}^{-1}
\hht_\mu[\Xi_\ga]\nonumber\\
&=
\sum_{|\mu|\leq|\la|}
K_{(n-|\la|,\la)(n-|\mu|,\mu)}^{-1}
\left< h_{n-|\mu|} h_\mu, p_\ga \right>\nonumber\\
&= \left<s_{(n-|\la|,\la)}, p_\ga\right> = \chi^{(n-|\la|,\la)}(\ga)~.\label{eq:chval}
\end{align}
If $n\geq|\lambda| + \lambda_1$ then this last expression
is equal to the value of the irreducible character of the symmetric group indexed by $(n-|\lambda|,\lambda)$ 
evaluated at an element of cycle type $\ga$.

This allows us to state a first characterization of
the symmetric functions $\st_\lambda$.
\begin{theore} (Part (1))
For a fixed partition $\lambda$, $\st_\la$ is
the unique symmetric function with the property that
for all $n \geq |\la|+\la_1$ and for all partitions $\gamma$ of $n$,
$$\st_\lambda(x_1, x_2, \ldots, x_n) = \chi^{(n-|\la|,\la)}(\ga),$$
where $x_1, x_2, \ldots, x_n$ are the eigenvalues of a permutation matrix of cycle structure 
$\ga$ and $\chi^{(n-|\la|,\la)}(\ga)$
are the values of the irreducible characters of the symmetric group.
\end{theore}

\begin{proof} By equation \eqref{eq:chval} we have
established that for all $n \geq |\la|+\la_1$ and
for any partition $\gamma$ of $n$,
\begin{equation*}
\st_\lambda[\Xi_\gamma] = \chi^{(n-|\la|,\la)}(\ga)~.
\end{equation*}
By Corollary \ref{cor:evalsf}, the only symmetric
function with this property must be equal to $\st_\lambda$.
\end{proof}

We call the basis $\{\st_\la\}$ the characters of the
irreducible representations of the symmetric group
when the symmetric group is realized as
permutation matrices. They are characters
in the same way that the Schur functions
are the characters of the irreducible representations of
the general linear group.  We therefore
name $\{\st_\lambda\}$ the {\it irreducible character
basis}.

An irreducible polynomial $GL_n$-module $W^\lambda$,
where $\lambda$ is a partition,  has character equal to
$s_\lambda(x_1, x_2, \ldots, x_n)$.  As we are considering
the embedding of $S_n$ as a subgroup of $GL_n$ we
may consider the decomposition of $W^\lambda$ into
irreducible $S_n$ modules when we restrict from
$GL_n$ to $S_n$.  The coefficients of the
expansion of $s_\lambda$ into irreducible character
basis establishes a second characterization of the
symmetric functions $\{\st_\la\}$.

\begin{theor} (Part (2)) The set $\{ \st_\lambda \}$ is
the unique family of symmetric functions such
that, for a sufficiently large $n$, if the multiplicity of
the $S_n$ module $S^{(n-|\mu|,\mu)}$ in the restriction
of the $GL_n$ module $W^\lambda$ to $S_n$ is $r_{\la\mu}$,
that is,
$$ W^\lambda\downarrow_{S_n}^{GL_n}
\simeq \bigoplus_{\mu}
(S^{(n-|\mu|,\mu)})^{\oplus r_{\la\mu}},$$
\hbox{ then }
$$s_\lambda = \sum_{\mu} r_{\la\mu} \st_{\mu}~.$$
\end{theor}

\begin{proof}
We first would like to show that $r_{\lambda\mu}$
is independent of $n$ for $n$ sufficiently large.
A theorem of Littlewood \cite{Lit, ScharfThibon} says
that $r_{\la\mu} = \left< s_\lambda,
s_{(n-|\mu|,\mu)}[1 + s_1+s_2+\cdots]\right>$.
The first consequence of this formula is that
$r_{\lambda\mu} = 0$ if $|\mu|>|\la|$ and $r_{\la\la} = 1$.

Next, choose $n>|\la|+\mu_1$.  Since
$$r_{\la\mu} = \sum_{r \geq 0}
\left< s_\lambda,
s_{(n-|\mu|,\mu)/(r)}[s_1+s_2+\cdots]\right>$$
then all the terms on the right hand side
are equal to $0$ unless
$r \geq n-|\lambda| > \mu_1$, otherwise the
degree of $s_{(n-|\mu|,\mu)/(r)}[s_1+s_2+\cdots]$ is
larger than the degree of $s_\lambda$.
If $r \geq n-|\la|\geq \mu_1$,
then $s_{(n-|\mu|,\mu)/(r)}[s_1+s_2+\cdots] =
s_d[s_1+s_2+ \cdots ] s_\mu[s_1+s_2+ \cdots]$
for some integer $d=n-r-|\mu|$.  We also have that
$$\left< s_\lambda,
s_d[s_1+s_2+ \cdots ] s_\mu[s_1+s_2+ \cdots]\right>=0$$
for $d>|\la|-|\mu|$, hence
$$r_{\la\mu} = \sum_{r \geq 0}
\left< s_\lambda,
s_{(n-|\mu|,\mu)/(r)}[s_1+s_2+\cdots]\right>
= \sum_{d=0}^{|\la|-|\mu|}
\left< s_\lambda,
s_d[s_1+s_2+ \cdots ] s_\mu[s_1+s_2+ \cdots]\right>~.$$

We conclude that the expression for $r_{\lambda\mu}$
is independent of $n$.  Hence for
$n$ sufficiently large and for all partitions
$\gamma$ such that $|\ga|\geq n$,
\begin{equation}
s_\la[\Xi_\ga] = \sum_{\mu} r_{\la\mu} \chi^\mu(\ga).
\end{equation}
By Theorem \ref{th:uniqueness} (part (1)), we know that this implies
$s_\la[\Xi_\ga] = \sum_{\mu} r_{\la\mu} \st_\mu[\Xi_\ga]$
and so by Corollary \ref{cor:evalsf} we
conclude that $s_\la = \sum_{\mu} r_{\la\mu} \st_\mu$.
These coefficients $r_{\la\mu}$ are then the
change of basis between the Schur functions and
the irreducible character basis.
\end{proof}

Recall that the Kronecker product is the bilinear
product on symmetric functions defined on the power sum basis by
\begin{equation}
\frac{p_\la}{z_\la} \ast \frac{p_\mu}{z_\mu} =
\delta_{\la=\mu} \frac{p_\la}{z_\la}~.
\end{equation}
The symbol $\delta_{\la=\mu}$ is the Kronecker delta function
that is equal to $1$ if $\la=\mu$ and $0$ otherwise.
Since $\left< \frac{p_\la}{z_\la}, p_\mu \right> = \delta_{\la=\mu}$,
we can verify the trivial calculation

\begin{equation}
\left< \frac{p_\la}{z_\la} \ast \frac{p_\mu}{z_\mu}, p_\ga \right>
= \delta_{\lambda=\mu} \delta_{\gamma=\lambda}
= \delta_{\la=\ga} \delta_{\mu=\ga}
= \left< \frac{p_\la}{z_\la} , p_\ga \right>
\left< \frac{p_\mu}{z_\mu}, p_\ga \right> ~.
\end{equation}
Since our product and scalar product
are bilinear, we have for any
symmetric functions $f$ and $g$,
\begin{equation}
\left< f \ast g, p_\ga \right>
= \left< f, p_\ga \right>
\left< g, p_\ga \right>~.
\end{equation}

By work of Murnaghan \cite{Murg2, Murg3}, for arbitrary
partitions $\la, \mu$ and $\nu$, there exists coefficients
$\gb_{\la\mu}^\nu$
with the property that
\begin{equation}
s_{(n-|\la|,\la)} \ast s_{(n-|\mu|,\mu)} = \sum_{\nu} \gb_{\la\mu}^\nu s_{(n-|\nu|,\nu)}
\end{equation}
for all sufficiently large $n$.
The $\gb_{\la\mu}^\nu$ are usually
referred to as `reduced' or `stable' Kronecker
coefficients (see for example \cite{EA, BOR, Gu, Klyachko}).

As we show next, when we write $\st_\lambda \st_\mu$ as a linear combination of the $\st_\nu$, 
we find the structure constants are the stable Kronecker coefficients.

\begin{theorem}  For partitions $\la$ and $\mu$,
\begin{equation}
\st_\la \st_\mu =
\sum_{|\nu| \leq |\la|+|\mu|} \gb_{\la\mu}^\nu \st_\nu~.
\end{equation}
\end{theorem}

\begin{proof}
We begin by evaluating the product of these functions at
the eigenvalues of a permutation matrix.
\begin{align*}
\st_\la[\Xi_\ga] \st_\mu[\Xi_\ga] &=
\left< s_{(|\ga|-|\la|,\la)}, p_\ga \right>
\left< s_{(|\ga|-|\mu|,\mu)}, p_\ga \right>\\
&= \left< s_{(|\ga|-|\la|,\la)}\ast s_{(|\ga|-|\mu|,\mu)}, p_\ga \right>\\
&= \sum_{|\nu| \leq |\la|+|\mu|}
\gb_{\la\mu}^\nu \left< s_{(|\ga|-|\nu|,\nu)}, p_\ga \right>\\
&= \sum_{|\nu| \leq |\la|+|\mu|}
\gb_{\la\mu}^\nu \st_\nu[\Xi_\ga]~.
\end{align*}

Since this expression is an
identity for all $\ga$ of sufficiently large size, we conclude by
Corollary \ref{cor:evalsf} that the theorem holds.
\end{proof}

The symmetric group character basis are not the only
symmetric functions which will have the $\gb_{\la\mu}^\nu$
coefficients as their structure coefficients.  Indeed,
any algebra isomorphism applied to the irreducible character
basis will have the same coefficients in their product
expansion.  However, if we also specify what the
symmetric functions are equal to for a family of
generators then this determines the functions
uniquely.  In the following characterization we
choose to specify the elements $\st_{1^r}$ for $r \geq 1$
(since this is the most compact expression we could identify).

\begin{theo} (Part (3))
The set $\{ \st_\lambda \}$ is
the family of symmetric functions such
that for $r \geq 1$, $s_{1^r} = \st_{1^r} + \st_{1^{r-1}}$
(or, equivalently, $\st_{1^r} = \sum_{i=0}^r (-1)^i e_{r-i}$) and
\begin{equation}\label{eq:structure}
\st_\la \st_\mu = \sum_{\nu}
{\overline g}_{\la\mu}^\nu \st_\ga~.
\end{equation}
\end{theo}
\begin{proof}  Assume that $\st_\la$ is a family of symmetric
functions that satisfy equation \eqref{eq:structure} and
such that $s_{1^r} = \st_{1^r} + \st_{1^{r-1}}$.

It is known that the coefficients
${\overline g}_{\la\mu}^\nu = c_{\la\mu}^\nu$
for $|\la|+|\mu| = |\nu|$ (see for instance \cite{BOR,Lit}).  Let $\ga$ be the partition
$\lambda$ with the first column removed (that is
$\ga = (\la_1-1, \la_2-1,\ldots, \la_{\ell(\la)}-1)$).
It follows that $\st_{1^{\ell(\la)}} \st_\ga$ is equal
to $\st_\lambda$ plus other terms which are indexed by
partitions which are either larger than $\lambda$ in
dominance order or of smaller size than $|\lambda|$.
This implies that there is an order where $\st_\lambda$
is uni-triangularly related to the elements of the form
$\st_{1^{\la_1'}} \st_{1^{\la_2'}}
\cdots \st_{1^{\la_{\ell(\la)}'}}$
(where $\lambda'_i$ is the length of the
$i^{th}$ column of $\lambda$)
and hence $\st_\lambda$ are a linearly independent
set of elements which are determined by products of $\st_{1^r}$.
A choice of the
initial condition of the value
of $\st_{1^r}$ as an element of the
symmetric functions for $r \geq 0$
determines the embedding of the basis $\st_\lambda$
as elements of the symmetric functions.

The fact that $s_{1^r} = \st_{1^r} + \st_{1^{r-1}}$
follows by computing the character of $\bigwedge^r(V_n)$
(the $r^{th}$ exterior product of the $S_n$ permutation module $V_n$) both
as a $GL_n$ character and as a $S_n$ character.
This initial condition determines
the family of symmetric functions for all partitions.
\end{proof}

\begin{remark}
For this third characterization of the irreducible character
basis we could have equally defined it to be the basis
which satifies equation \eqref{eq:structure} and which
satisfies an initial condition on almost any set of
generators of the algebra of symmetric functions
(e.g. $\st_{r}$ for $r \geq 0$).  We chose to state
it in terms of $\st_{1^r}$ because the expression was
the most compact to state.
\end{remark}

\end{section}

\begin{section}{The irreducible character expansion of a complete symmetric function}\label{sec:handetost}
The main result of this section is an expansion of the complete homogeneous functions $h_\lambda$
in terms of the $\st$-basis.  The coefficients are described combinatorially by combining the notion 
of multiset partition of a multiset
and column strict tableau. 
To work with column strict tableaux
on sets or multisets we need to establish a total order on
these objects.  We remark that, with few restrictions, we can
do this with almost any total order and so we will
use lexicographic if we read the entries of the multiset in
increasing order.  This may mean that the tableaux we work
with will have content which is not a partition, but
this is typical with column strict tableaux.


If $\lambda = (\lambda_1, \lambda_2, \ldots, \lambda_{\ell(\lambda)})$ is a
partition, then use the
notation ${\overline \la} = (\la_2, \la_3,\ldots, \la_{\ell(\la)})$
to represent the partition with the first part removed.  For a tableau
$T$, let $shape(T)$ denote the partition of the outer shape of the tableau.

\begin{theorem} \label{thm:stexpofh}
For a partition $\mu$,
\begin{equation}
h_\mu = \sum_{T}
\st_{\overline {shape(T)}}
\end{equation}
the sum is over all skew-shape column strict tableaux of shape $\lambda/(\lambda_2)$
for some partition $\lambda$ where the cells are filled
with non-empty multisets of labels such that the total content of the
tableau is $\dcl1^{\mu_1},2^{\mu_2},\ldots,\ell^{\mu_\ell}\dcr$.
\end{theorem}

\begin{proof}
From equation \eqref{eq:definingrel},
we know the expansion of $h_\mu$
in terms of multiset partitions of a multiset and by
\eqref{eq:httost} we know the expansion of $\hht_{\mt(\pi)}$ in
the $\st_\la$ basis terms of column strict tableaux.
Combining these two expansions
we have that for an $n$ sufficiently large,
\begin{equation}
h_\mu = \sum_{\pi \mvdash
\dcl1^{\mu_1},2^{\mu_2},\ldots,\ell^{\mu_\ell}\dcr}
\sum_{\la \vdash n}
K_{\la(n-\ell(\pi),\mt(\pi))} \st_{\overline \la}~.
\end{equation}
Now we note that for every multiset partition $\pi$
and column strict tableaux of shape $\lambda$
and content given by the partition
$(n-\ell(\pi),\mt(\pi))$ we can create a skew-shaped
tableau whose entries are multisets by replacing
the $n-\ell(\pi)$ labels with a $1$ with a blank
so that it is of skew shape $\lambda/(n-\ell(\pi))$
and the other labels by
their corresponding multiset in $\pi$.  The value of
$n$ needs to be chosen  so that
$n-\ell(\pi)$ is larger than the
size of the first part of $\overline \la$.

To explain why
this is equal to the description stated in the theorem where there
are precisely $\lambda_2$ blank cells in the first row, we note
that
$K_{\la(n-\ell(\pi),\mt(\pi))} = K_{(n'-|{\overline \la}|,{\overline \la})
(n'-\ell(\pi),\mt(\pi))}$ as long as $n'-\ell(\pi) \geq \la_2$.  This is because
there is a bijection between these two sets of tableaux by inserting or deleting $1$s 
in the first row of each tableau in the set.
In particular, we may choose $n' - \ell(\pi)= \la_2$ and the description
of the tableaux are precisely those that are column strict of
skew of shape $(n'-|{\overline \la}|,{\overline \la})/(\lambda_2)$
and whose entries are the multisets in $\pi$.
\end{proof}

\begin{example}\label{ex:htost}
Consider the following 20 column strict tableaux of content
$\dcl1^2,2\dcr$.

\begin{equation*}
\young{1&1&2\cr&&\cr}\hskip .25in
\young{2\cr1&1\cr&\cr}\hskip .25in
\young{1&1\cr&&2\cr}\hskip .25in
\young{1&2\cr&&1\cr}\hskip .25in
\young{11&2\cr&\cr}\hskip .25in
\young{1&12\cr&\cr}\hskip .25in
\young{2\cr1\cr&1\cr}\hskip .25in
\young{2\cr11\cr\cr}\hskip .25in
\young{12\cr1\cr\cr}\hskip .25in
\young{1\cr&1&2\cr}
\end{equation*}
\begin{equation*}
\young{2\cr&1&1\cr}\hskip .25in
\young{1\cr&12\cr}\hskip .25in
\young{2\cr&11\cr}\hskip .25in
\young{12\cr&1\cr}\hskip .25in
\young{11\cr&2\cr}\hskip .25in
\young{\hbox{\Tiny 112}\cr\cr}\hskip .25in
\young{1&1&2\cr}\hskip .25in
\young{1&12\cr}\hskip .25in
\young{11&2\cr}\hskip .25in
\young{\hbox{\Tiny 112}\cr}
\end{equation*}
Theorem \ref{thm:stexpofh} states then that
\begin{equation}
h_{21} = \st_3 + \st_{21} + 4\st_2 +3\st_{11} +7\st_1 +4\st_{()}~.
\end{equation}
\end{example}

\begin{example}\label{ex:settableaux}
Let us compute the decomposition of $V^{\otimes 4}$
where $V = \CC\{ x_1, x_2, x_3, \ldots ,x_n \}$ as an $S_n$ module
with the diagonal action.  The module $V$ has character equal to
$\hht_1 = h_1$.  Therefore to compute the decomposition of this
character into $S_n$ irreducibles we are looking for the expansion
of $h_{1^4}$ into the irreducible character basis.

Using Sage \cite{sage,sage-co} we compute that it is
\begin{align}
h_{1^4} = &15\st_{()} + 37\st_{1} + 31\st_{11} + 10\st_{111} + \st_{1111} + 31\st_{2}\\
&+ 20\st_{21} + 3\st_{211} + 2\st_{22} + 10\st_{3} + 3\st_{31} + \st_{4}~.
\nonumber
\end{align}

If $n\geq6$ then the multiplicity of the irreducible $(n-3,3)$ will be $10$.
The combinatorial interpretation of this value is the number of column strict
tableaux with entries that
are multisets (or in this case sets) of $\{1,2,3,4\}$ of skew-shape $(4,3)/(3)$
or $(3,3)/(3)$. Those tableaux are
\begin{equation*}
\young{1&2&3\cr&&&4\cr}\hskip .2in
\young{1&2&4\cr&&&3\cr}\hskip .2in
\young{1&3&4\cr&&&2\cr}\hskip .2in
\young{2&3&4\cr&&&1\cr}
\end{equation*}
\begin{equation*}
\young{14&2&3\cr&&\cr}\hskip .2in
\young{13&2&4\cr&&\cr}\hskip .2in
\young{12&3&4\cr&&\cr}\hskip .2in
\young{1&23&4\cr&&\cr}\hskip .2in
\young{1&24&3\cr&&\cr}\hskip .2in
\young{1&2&34\cr&&\cr}
\end{equation*}
What is interesting about this example is that the usual combinatorial interpretation
for the repeated Kronecker product $(\chi^{(n-1,1)}+\chi^{(n)})^{\ast k}$ is
stated in other places in the literature in terms of vascillating tableaux \cite{CG, HL, BHH, BH, BDE}.
Thus, this special case of Theorem \ref{thm:stexpofh} gives a new combinatorial description
of the multiplicities in terms of set valued tableaux.
\end{example}
\end{section}

\begin{section}{Character polynomials and the irreducible
character basis}\label{sec:stcharpoly}

Following the notation of \cite{GG},
a character polynomial is a multivariate
polynomial $q_\lambda(x_1, x_2, x_3, \ldots)$ in the variables
$x_i$ such that for integer values $x_i = m_i \in {\mathbb Z}$
with $m_i \geq 0$ and $n = \sum_{i \geq 1} i m_i$,
\begin{equation}
q_\lambda(m_1, m_2, m_3, \ldots)
= \chi^{(n-|\lambda|,\lambda)}(1^{m_1}2^{m_2}3^{m_3}\cdots)
\end{equation}
as long as $n$ is larger than or equal to $|\la|+\la_1$.

Character polynomials were first used
by Murnaghan \cite{Murg}. Much later, Specht
\cite{Sp} gave determinantal formulas and expressions in
terms of binomial coefficients for these polynomials.
They are treated as an example
in Macdonald's book \cite[ex. I.7.13 and I.7.14]{Mac}.
More recently, Garsia and Goupil \cite{GG}
gave an umbral formula for computing them.  We will
show in this section that character polynomials
are a transformation of character symmetric functions
and this will allow us to give an expression
for character symmetric functions in the power sum basis.

As a consequence
of Lemma \ref{lem:polyzeros} and Proposition
\ref{prop:rootsimplsf} we have the following
relationship between the character polynomials
$q_\lambda( x_1, x_2, \ldots)$
and character basis $\st_\la$.

\begin{prop} \label{prop:charpolyst} For $n\geq 0$ and
$\lambda \vdash n$,
$$q_\lambda(x_1, x_2, x_3, \ldots) = \st_\la \Big|_{p_k \rightarrow \sum_{d|k} d x_d}$$
and
$$\st_\la = q_\lambda(x_1, x_2, x_3, \ldots) \Big|_{x_k \rightarrow \frac{1}{k}\sum_{d|k} \mu(k/d) p_d }$$
where $\Big|_{a_i\rightarrow b_i}$ means that we are replacing $a_i$ with the expression $b_i$.
\end{prop}

In \cite{GG}, the character polynomials are computed algorithmically.  If we make an additional substitution, i.e.,  $x_k$ by $\frac{1}{k} \sum_{d|k} \mu(k/d) p_d$, in their algorithm, then we obtain $\st_\la$ using the following steps.
\begin{enumerate}
\item Expand the Schur function $s_\lambda$ in the power sums basis
$s_\la = \sum_{\ga} \frac{\chi^\lambda(\ga)}{z_\ga} p_\ga$.
\item Replace each power sum $p_i$ by $i x_i -1$.
\item Expand each product $\prod_i (ix_i-1)^{a_i}$
as a sum $\sum_g c_g \prod_i x_i^{g_i}$.
\item Replace each $x_k^{g_k}$ by $(x_k)_{g_k} =
x_k(x_k-1)\cdots(x_k-g_k+1)$.
\item Replace each $x_k$ by $\frac{1}{k} \sum_{d|k} \mu(k/d) p_d$.
\end{enumerate}

\begin{example}  To compute $\st_3$ we follow the steps to obtain:
\begin{enumerate}
\item $s_3 = \frac{1}{6}(p_{1^3} + 3 p_{21} + 2 p_3)$
\item $\frac{1}{6}(p_{1^3} + 3 p_{21} + 2 p_3) \rightarrow
\frac{1}{6}((x_1-1)^3+3(2x_2-1)(x_1-1)+2(3x_3-1))$
\item $\frac{1}{6}((x_1-1)^3+3(2x_2-1)(x_1-1)+2(3x_3-1))=
\frac{1}{6}x_1^3-\frac{1}{2}x_1^2+x_1x_2-x_2+x_3$
\item $q_{3} = \frac{1}{6}(x_1)_3 - \frac{1}{2}(x_1)_2
+ x_1x_2 - x_2+ x_3$\label{step}
\item $\st_{3} = \frac{1}{6}(p_1)_3 - \frac{1}{2}(p_1)_2
+ p_1\frac{p_2-p_1}{2} - \frac{p_2-p_1}{2}+ \frac{p_3-p_1}{3}$
\end{enumerate}
\end{example}

As an important consequence, we derive a power sum expansion of the irreducible
character basis by following the algorithm stated above.

\begin{theorem} \label{thm:pexpansionofst} For $n\geq 0$ and
$\lambda \vdash n$,
\begin{equation}\label{eq:pexpansionofst}
\st_\lambda = \sum_{\gamma \vdash n}
\chi^{\lambda}(\gamma) \frac{\bfp_{\ga}}{z_\gamma}
\end{equation}
that is, the linear map defined by $\Gamma( s_\lambda ) = \st_\lambda$
has the property that $\Gamma(p_\ga) = \bfp_{\ga}$ where
\begin{equation}\label{eq:weirdp}
\bfp_{i^r} = \sum_{k=0}^r (-1)^{r-k} i^k \binom{r}{k}
\left(\frac{1}{i} \sum_{d|i} \mu(i/d) p_d \right)_k\hbox{ and }\bfp_{\gamma}
:= \prod_{i \geq 1} {\mathbf p}_{i^{m_i(\gamma)}}~,
\end{equation}
and $(x)_k$ denotes the $k$-th falling factorial.
\end{theorem}

\begin{proof}
The proof of this proposition is exactly the steps outlined
in the result of \cite{GG} and then add the additional
step of replacing $x_i$ with $\frac{1}{i} \sum_{d|i} \mu(i/d) p_d$
in step \eqref{step}.
The Schur function has a power sum expansion given by
\begin{equation}
s_\lambda = \sum_{\ga \vdash |\la|} \chi^{\lambda}(\ga)\frac{p_\ga}{z_\ga}
= \sum_{\ga \vdash |\la|}
\chi^{\lambda}(\ga)\frac{1}{z_\ga} \prod_{i=1}^{\ell(\ga)} (p_i)^{m_i(\ga)}~.
\end{equation}
Then in the next step we replace $p_i$ with $i x_i -1$ and expand the expression.
The part of the expression $(p_i)^{r}$ becomes
\begin{equation} \label{eq:interstep}
(p_i)^{r}\Big|_{p_i \rightarrow ix_i-1} =
(i x_i -1)^{r} = \sum_{k=0}^r (-1)^{r-k} i^k \pchoose{r}{k} x_i^k~.
\end{equation}
In the last step the \cite{GG} algorithm replaces $x_i^k$ with $(x_i)_k$
and that is the expression for the character polynomial.  To recover the
irreducible character function $\st_\la$,
we use equation \eqref{eq:recoversf} and replace
$x_i$ with $\frac{1}{i} \sum_{d|i} \mu(i/d) p_d$. Replacing $x_i^k$ in \eqref{eq:interstep}
with $\left(\frac{1}{i} \sum_{d|i} \mu(i/d) p_d \right)_k$ means that
$(p_i)^r$ will be replaced by the expression in equation \eqref{eq:weirdp}.

It follows that the composition of these steps changes $s_\lambda$ to
$\st_\lambda$ and the power sum expansion of the Schur function to
the right hand side of equation \eqref{eq:pexpansionofst}.
\end{proof}

We also present a similar formula for the expansion of the
induced trivial character basis in the power sum basis.  To do this we
introduce a basis which acts like an indicator function for
evaluation at the roots of unity.  Define
\begin{equation}\label{eq:otherweirdp}
\bfpb_{i^r} = i^r
\left(\frac{1}{i} \sum_{d|i} \mu(i/d) p_d \right)_r\hbox{ and }\bfpb_{\gamma}
:= \prod_{i \geq 1} \bfpb_{i^{m_i(\gamma)}}~.
\end{equation}

\begin{lemma} \label{lem:evalowp}
For partitions $\ga$ and $\mu$ such that $|\mu|<|\ga|$,
then $\bfpb_\ga[\Xi_\mu] = 0$.  Moreover, if $|\mu|=|\ga|$, then
\begin{equation}
\bfpb_\ga[\Xi_\mu] = z_\ga \delta_{\ga=\mu}~.
\end{equation}
\end{lemma}

\begin{proof} Since $p_d[\Xi_\mu] = \sum_{d'|d} d' m_{d'}(\mu)$ and
\begin{equation}
\frac{1}{i} \sum_{d|i} \mu(i/d) \left( \sum_{d'|d} d' m_{d'}(\mu) \right) = m_i(\mu)~,
\end{equation}
then plugging into \eqref{eq:otherweirdp} we see
\begin{equation}
\bfpb_{i^r}[\Xi_\mu] = i^r ( m_i(\mu) )_r~.
\end{equation}
and hence $\bfpb_\ga[\Xi_\mu] = \prod_{i \geq 1} i^{m_i(\ga)} ( m_i(\mu) )_{m_i(\ga)}$.

Now if $|\mu|<|\ga|$ or $|\mu|=|\ga|$ and $\mu\neq\ga$,
then there exists at least one $i$ such that $m_i(\mu)<m_i(\ga)$
and for that value $i$, $( m_i(\mu) )_{m_i(\ga)}=0$ and hence $\bfpb_\ga[\Xi_\mu]=0$.

If $\mu = \ga$, then $\bfpb_\ga[\Xi_\ga] = \prod_{i \geq 1} i^{m_i(\ga)} ( m_i(\ga) )_{m_i(\ga)}
= \prod_{i \geq 1} i^{m_i(\ga)} m_i(\ga)! = z_\ga$.
\end{proof}

This basis can then be used to give a formula for the
induced trivial character basis.
\begin{prop} For $n\geq 0$ and $\la \vdash n$,
\begin{equation}\label{eq:hhtpowerexp}
\hht_\la = \sum_{\ga \vdash n} \left< h_\la, p_\ga \right> \frac{\bfpb_\ga}{z_\ga},
\end{equation}
that is, the linear map defined by $\Theta( h_\la) = \hht_\la$ has the property that
$\Theta(p_\ga) = \bfpb_\ga$.
\end{prop}
\begin{proof}
For any partition $\mu$ such that $|\mu|<|\la|$, then
$\hht_\la[\Xi_\mu]=\left< h_{|\mu|-|\la|} h_\la,
p_\mu \right> = 0$ because $|\mu|-|\la|<0$.
In addition, by Lemma \ref{lem:evalowp},
$\sum_{\ga \vdash |\la|} \left< h_\la, p_\ga \right> \frac{\bfpb_\ga[\Xi_\mu]}{z_\ga}=0$.
If $|\mu|=|\la|$, then $\left< h_{|\mu|-|\la|} h_\la, p_\mu \right>
= \left< h_\la, p_\mu \right>$ and
\begin{equation}
\sum_{\ga \vdash |\la|} \left< h_\la, p_\ga \right> \frac{\bfpb_\ga[\Xi_\mu]}{z_\ga}
= \left< h_\la, p_\mu \right> = \hht_\la[\Xi_\mu]~.
\end{equation}
We can conclude since we have equality for all evaluations at
$\Xi_\mu$ for $|\mu|\leq|\la|$, then
by Proposition \ref{prop:rootsimplsf}, Equation
\eqref{eq:hhtpowerexp} holds at the level of symmetric functions.
\end{proof}

\begin{remark}
After finding this formula for $\hht_\lambda$ in terms
of the elements $\bfpb_\ga$ we noticed that a similar
expression appears in Macdonald's book \cite{Mac} on page 121.
He defines polynomials for each partition $\rho$
in variables $a_1, a_2, \ldots$
as $\binom{a}{\rho} := \prod_{r \geq 1} \binom{a_r}{m_r(\rho)}$.
We noticed that if $a_i = \frac{1}{i} \sum_{d|i} \mu(i/d)p_d$
then $\binom{a}{\rho} = \frac{\bfpb_\rho}{z_\rho}$.  The
interested reader can translate Equation (4) from Example 14
on page 123 for the character polynomial to conclude that
\begin{equation}\label{eq:Macform1}
\st_\la = \sum_{\sigma,\rho} (-1)^{\ell(\sigma)}
\left< s_\la, p_\rho p_\sigma \right>
\frac{\bfpb_\rho}{z_\sigma z_\rho}
\end{equation}
summed over all partitions $\rho$ and $\sigma$ such that
$|\rho|+|\sigma|=|\lambda|$.  We can also translate Equation
(5) from the same example on page 124 to show that
\begin{equation}\label{eq:Macform2}
\st_\la = \sum_{\mu} (-1)^{|\la|-|\mu|} \sum_{\ga\vdash|\mu|} \left< s_\mu, p_\ga \right>
\frac{\bfpb_\ga}{z_\ga}
\end{equation}
where the outer sum is over partitions $\mu$ such that
$\la/\mu$ is a vertical strip (no more than one cell in each row of the
skew partition).
\end{remark}
\end{section}

\begin{section}{The partition algebra and a Murnaghan-Nakayama rule}
\label{sec:MNrule}

The partition algebra was independently defined in the work of Martin \cite{Ma1,Ma2,Ma3,Ma4} and Jones \cite{Jo}.
Jones showed that the partition algebra  is the centralizer algebra of the diagonal action of the
symmetric group on tensor space. In other words he described a Schur-Weyl duality between the symmetric
group and the partition algebra. Later,  Halverson \cite{Hal}
described the analogues of the Frobenius formula and the
Murnaghan-Nakayama rule to compute the characters of the partition algebra.  In this section we describe a
connection between the partition algebra characters and our irreducible character basis.

If $V$ is the $r$-dimensional defining representation of $S_r$, then the centralizer of the
diagonal action on $V^{\otimes n}$ depends on the two parameters $n$ and $r$ and is denoted
by $P_n(r)$.  The irreducible characters are indexed by partitions $(r - |\la|,\la)$, where $\la$
is a partition of size less than or equal to $n$. Halverson described conjugacy class analogues
and denoted the representatives of this classes by $d_\mu$, where $\mu$ is a partition of size
less than or equal to $n$. Using these notations, the irreducible partition algebra character
values are denoted by  $\chi_{P_n(r)}^{(r - |\la|,\la)}(d_\mu)$.

Corollary 4.2.3 of \cite{Hal} states the following properties
of the partition algebra characters.
\begin{cor} \label{cor:Hal1} If $|\la|\leq n$ and $\mu$ is a composition of size less than or equal to $n$, then
\begin{enumerate}
\item $\chi_{P_n(r)}^{(r - |\la|,\la)}(d_\mu) = \begin{cases}
r^{n-|\mu|}\chi_{P_{|\mu|}(r)}^{(r - |\la|,\la)}(d_\mu)
&\hbox{ if } |\mu|\geq|\la|,\\
0&\hbox{ if }|\mu|<|\la|
\end{cases}$
\item $|\mu|=|\la|=n$, then
$\chi_{P_n(r)}^{(r-|\la|,\la)}(d_\mu) =
\chi_{S_n}^{\la}(\mu)$
\item if $r \geq 2n$ and $|\mu|=n$, then
$\chi_{P_n(r)}^{(r - |\la|,\la)}(d_\mu)$ is independent
of $r$.
\end{enumerate}
\end{cor}

For a positive integer $n$ and
$\mu \vdash n$, the usual Frobenius formula for the symmetric functions is a consequence of the classical Schur-Weyl duality between $S_n$ and $GL_r$, it states
\begin{equation}
p_\mu = \sum_{\lambda \vdash n}
\chi^{\lambda}_{S_n}(\mu) s_\lambda
\end{equation}
where $s_\lambda$ (as symmetric functions)
are the irreducible $GL_r$ characters.

If we restrict the diagonal action of $GL_r$ to the symmetric group, $S_r$, realized by the permutation matrices,
we obtain the Schur-Weyl duality between the symmetric group $S_r$ and the partition algebra $P_n(r)$.
A decomposition of $V^{\otimes n}$ as a $(P_n(r), S_r)$-module into irreducibles yields the analogue of the
Frobenius formula for $P_n(r)$ and the symmetric group.  See  equation (3.2.1) and Theorem 3.2.2
of \cite{Hal} where we assume that $r\geq 2|\mu|$, and $\ga \vdash r$,
\begin{equation}
p_\mu[\Xi_\gamma] = \sum_{|\lambda| \leq |\mu|}
\chi^{(r-|\lambda|,\lambda)}_{P_{|\mu|}(r)}(d_\mu)
\st_\lambda[\Xi_\gamma]~.
\end{equation}

Since this identity hold for all $r$ greater than a fixed value and all
partitions $\gamma \vdash r$,
then by Corollary \ref{cor:evalsf}, this expression is a symmetric
function identity and we have
\begin{equation} \label{eq:Frobform}
p_\mu = \sum_{|\lambda| \leq |\mu|}
\chi^{(r-|\lambda|,\lambda)}_{P_{|\mu|}(r)}(d_\mu)
\st_\lambda~.
\end{equation}

Our next result is a statement which is equivalent to the
Murnaghan-Nakayama rule for the computation of the
irreducible symmetric group characters.  To state how this
relation appears in the irreducible character basis, we first introduce a little notation.

For partitions $\lambda$ and $\nu$ such that $\nu \subseteq \lambda$,
we will say that $\lambda$ differs from $\nu$ by a
$k$ border strip (abbreviated $\la/\nu \in B_k$) if the skew partition $\la/\nu$
consists of $k$ cells which are connected and do not contain a
$2\times 2$ sub-configuration of cells.  When we think of these partitions
as having an extra long row, we will write $\la/\!\!/\nu \in B_k$ if
$\la/\nu \in B_k$ or if $\la=\nu$ (in which case we think of
the skew partition as $(r,\lambda)/(r-k,\lambda)$ for a sufficiently large $r$).
Let $ht(\la/\!\!/\nu)$ equal to the number of rows occupied by the
skew partition minus 1 and, in particular, $ht(\la/\!\!/\la)=0$.

\begin{lemma} \label{lem:stdualMNrule} For $n,k>0$ and $\la \vdash n$,
let $\mu \vdash 2n+k$.
\begin{equation}
\st_\la[\Xi_{(k,\mu)}] =
\sum_{\nu : \la/\!\!/\nu \in B_k}
(-1)^{ht(\la/\!\!/\nu)}\st_\nu[\Xi_\mu]
= \st_\la[\Xi_\mu]
+ \sum_{\nu : \la/\nu \in B_k}
(-1)^{ht(\la/\nu)}\st_\nu[\Xi_\mu]~.
\end{equation}
\end{lemma}
\begin{proof}
Recall that the Murnaghan-Nakayama rule says that
\begin{equation}
p_k s_\lambda = \sum_{\nu/\la \in B_k} (-1)^{ht(\nu/\la)}
s_\nu
\end{equation}
and similarly by duality,
\begin{equation}
p_k^\perp s_\lambda = \sum_{\la/\nu \in B_k}(-1)^{ht(\la/\nu)} s_\nu~.
\end{equation}
where $p_k^\perp$ denotes the adjoint to $p_k$ with respect to the inner product.

Next we calculate by translating the evaluation of
$\st_\la[\Xi_{(k,\mu)}]$ to a symmetric function scalar product.
\begin{align}
\st_\la[\Xi_{(k,\mu)}] &= \chi^{(|\mu|+k-|\la|,\la)}(k,\mu)\nonumber\\
&=\left< s_{(|\mu|+k-|\la|,\la)}, p_\mu p_k \right>\nonumber\\
&=\left< p_k^\perp s_{(|\mu|+k-|\la|,\la)}, p_\mu \right>\nonumber\\
&=\sum_{\nu : \la/\!\!/\nu \in B_k}
(-1)^{ht(\la/\!\!/\nu)}
\left< s_{(|\mu|-|\nu|,\nu)}, p_\mu \right>\nonumber\\
&=\sum_{\nu : \la/\!\!/\nu \in B_k}
(-1)^{ht(\la/\!\!/\nu)}\st_\nu[\Xi_\mu]~.\label{eq:stMNrule}
\end{align}

Now to complete the statement of the lemma, we note that
a $k$-border strip that starts in the first row of $(|\mu|+k-|\nu|,\nu)$
lies only in the first row because we assume that $|\mu|>2|\nu|.$
Hence one of the terms where $\la/\!\!/\nu$ is a $k$ border strip is
$\nu = \la$.  All the others partitions such that
$\la/\!\!/\nu$ is a $k$ border strip will have
the $k$ border strip start in the second row or higher and
in this case $\la/\nu$ will be a $k$-border strip.  Therefore
equation \eqref{eq:stMNrule} is equal to
\begin{equation}
= \st_\la[\Xi_\mu]
+ \sum_{\nu : \la/\nu \in B_k}
(-1)^{ht(\la/\nu)}\st_\nu[\Xi_\mu]~. \qedhere
\end{equation}
\end{proof}

Let $\la$ and $\nu$ be partitions
and choose $r>max(|\la|+\la_1, |\nu|+\nu_1)$.
Then $(r-|\la|,\la)$ and $(r-|\nu|,\nu)$ are partitions
and we may compute
\begin{align} \label{eq:orth}
\sum_{\mu \vdash r} \frac{1}{z_\mu} \st_\la[\Xi_\mu]
\st_\nu[\Xi_\mu]
&= \sum_{\mu \vdash r} \frac{1}{z_\mu}
\chi^{(r-|\la|,\la)}(\mu)
\chi^{(r-|\nu|,\nu)}(\mu)\nonumber\\
&=\frac{1}{r!} \sum_{\sigma \in S_r}
\chi^{(r-|\la|,\la)}(\sigma)
\chi^{(r-|\nu|,\nu)}(\sigma)~.
\end{align}
By the orthogonality
of symmetric group characters, this sum is equal to $1$ if $\la = \nu$
and $0$ otherwise.

This computation can be used to
obtain a single coefficient of an $\st_\lambda$ in
a symmetric function expression $f$
which we rephrase as the following lemma.

\begin{lemma} \label{lem:takecoeff}
Let $r$ be a positive integer such that
$r>2 deg(f)$. The coefficient of $\st_\lambda$ in $f$ is equal to
$\sum_{\mu \vdash r}
\frac{1}{z_\mu} \st_\la[\Xi_\mu]
f[\Xi_\mu]$~.
\end{lemma}

\begin{proof}
If $f = \sum_{\la} c_\la \st_\la$,
then for all $\la$ in the support of $f$, $(r-|\la|,\la)$
will always be a partition.  Therefore
since Equation \eqref{eq:orth} is equal to $1$
if $\lambda = \nu$ and $0$, otherwise, then by linearity
$\sum_{\mu \vdash r}
\frac{1}{z_\mu} \st_\la[\Xi_\mu]
f[\Xi_\mu] = c_\lambda$.
\end{proof}

 We will apply this lemma to obtain
the following expression which is equivalent
to Halverson's \cite{Hal} Murnaghan-Nakayama
rule for partition algebra characters.

\begin{theorem}
\label{thm:HalstMNrule}
For $k>0$ and a partition $\lambda$,
\begin{equation}
p_k \st_\lambda =
\sum_{\nu} \left(\sum_{d|k}
\sum_{\al }
(-1)^{ht(\la/\!\!/\al)+ht(\nu/\!\!/\al)} \right) \st_\nu
\end{equation}
where the inner sum is over all partitions
$\alpha$ such that
both $\la/\!\!/\al$ and $\nu/\!\!/\al$ are border strips
of size $d$ and the outer sum is over all partitions $\nu$
of size less than or equal to $k+|\lambda|$.
\end{theorem}

\begin{proof}  Our proof follows the computation of Halverson \cite{Hal},
but in the language of symmetric functions using the irreducible character basis. We can apply Lemma
\ref{lem:takecoeff} to take the coefficient of $\st_\nu$ in
$p_k \st_\la$.  To begin, we note that
$p_k[\Xi_\mu] = \sum_{d|k} d m_d(\mu)$ and
choose an $r$ sufficiently large.  We compute the coefficient
by the expression
\begin{align}
\sum_{\mu \vdash r}
\frac{1}{z_\mu} \st_\nu[\Xi_\mu]
p_k[\Xi_\mu] \st_\la[\Xi_\mu]
&=\sum_{d|k} \sum_{\mu \vdash r}
\frac{d m_d(\mu)}{z_\mu} \st_\nu[\Xi_\mu]
\st_\la[\Xi_\mu]~.
\end{align}

Now the non-zero terms of the sum over $\mu$
occur when $m_d(\mu)>0$ and
in this case $\frac{d m_d(\mu)}{z_\mu} = \frac{1}{z_{\mu-(d)}}~.$
This is equivalent to summing over all partitions $\mub$
of size $r-d$ and $\mu = (d,\mub)$.

Therefore by applying Lemma \ref{lem:stdualMNrule},
\begin{align}
\sum_{\mu \vdash r}
\frac{1}{z_\mu} \st_\nu[\Xi_\mu]
p_k[\Xi_\mu] \st_\la[\Xi_\mu]
&= \sum_{d|k} \sum_{\mub \vdash r-d}
\frac{1}{z_\mub} \st_\nu[\Xi_{(d,\mub)}]
\st_\la[\Xi_{(d,\mub)}]\nonumber\\
&= \sum_{d|k} \sum_{\mub \vdash r-d}
\frac{1}{z_\mub}
\sum_{\al }
\sum_{\be}
(-1)^{ht(\la/\!\!/\al)+ht(\nu/\!\!/\be)}
\st_\al[\Xi_\mub]
\st_\be[\Xi_\mub]\nonumber\\
&= \sum_{d|k}
\sum_{\al }
\sum_{\be}
(-1)^{ht(\la/\!\!/\al)+ht(\nu/\!\!/\be)}
\sum_{\mub \vdash r-d}
\frac{1}{z_\mub}
\st_\al[\Xi_\mub]
\st_\be[\Xi_\mub]\nonumber\\
&= \sum_{d|k}
\sum_{\al }
(-1)^{ht(\la/\!\!/\al)+ht(\nu/\!\!/\al)}
\end{align}
where the sum over $\al$ is such that
$\la/\!\!/\al$ is a border strip of size $d$
and the sum over $\beta$ is
$\nu/\!\!/\be$ is a border strip of size $d$.  The last equality holds
by the orthogonality relations on the symmetric group characters.
\end{proof}

An expansion of Theorem \ref{thm:HalstMNrule} using the
second equality in Lemma \ref{lem:stdualMNrule} yields the
following alternate expression for the Murnaghan-Nakayama rule
for the irreducible character basis.
\begin{cor}
For $k>0$,
\begin{align*}
p_k \st_\lambda =
&divisors(k)~\st_\la +
\sum_{d|k}
\sum_{\nu}
\sum_{\substack{\al:\nu/\al \in B_d\\
\la/\al \in B_d}}
(-1)^{ht(\la/\al)+ht(\nu/\al)} \st_\nu\\
&+\sum_{d|k}
\sum_{\nu:\la/\nu \in B_d} (-1)^{ht(\la/\nu)} \st_\nu
+\sum_{d|k}
\sum_{\nu:\nu/\la \in B_d} (-1)^{ht(\nu/\la)} \st_\nu
\end{align*}
where $divisors(k)$ is equal to the number of divisors of $k$.
\end{cor}
\end{section}

\begin{section}{Appendix I: Evaluations of symmetric functions at roots of unity}
\label{sec:specializations}

We refer readers to \cite[Lemma 5.10.1]{Lascoux} for the following
result which we will use as our starting point for
evaluations of symmetric functions.  In the following expressions, the notation $r|n$ is shorthand
for ``$r$ divides $n$.''

\begin{prop} \label{prop:evalhnXr}
For $r \geq 0$, $h_0[\Xi_r] = e_0[\Xi_r]=p_0[\Xi_r]=1$.
In addition, for $n>0$,
\begin{equation}
h_n[\Xi_r] = \delta_{r|n}, \hskip .2in
p_n[\Xi_r] = r \delta_{r|n}, \hskip .2in
e_n[\Xi_r] = (-1)^{r-1}\delta_{r=n}~.
\end{equation}
\end{prop}

We will need to take this further and give a combinatorial
interpretation for the evaluation of $h_\lambda[\Xi_\mu]$
in order to make a
connection with character symmetric functions.


\begin{prop}
For a nonnegative integer $n$ and a partition $\mu$,
$h_n[\Xi_\mu]$ is equal to the number of weak compositions
$\alpha$
of size $n$ and length $\ell(\mu)$ such that
$\mu_i$ divides $\alpha_i$ for all $i$.
\end{prop}

\begin{proof}
The alphabet addition formula for $h_n$ says
\begin{equation}
h_{n}[X_1,X_2,\cdots,X_r] =
\sum_{\substack{\alpha \models_w n\\
\ell(\al)=r}}
h_{\alpha_1}[X_1] h_{\alpha_2}[X_2]
\cdots h_{\alpha_r}[X_r]
\end{equation}
therefore evaluating at $\Xi_\mu$, we have
\begin{equation}\label{eq:hnatXimu}
h_n[\Xi_\mu] =
\sum_{\substack{\alpha \models_w n\\
\ell(\alpha) = \ell(\mu)}}
\prod_{i=1}^{\ell(\mu)} h_{\alpha_i}[\Xi_{\mu_i}]
\end{equation}
where the sum is over all weak compositions of $n$ (that
is, $0$ parts are allowed) such that the length of
the composition (including the $0$ parts) is equal
to the length of the partition $\mu$.
Since $h_{\alpha_i}[\Xi_{\mu_i}]=0$ unless $\mu_i | \alpha_i$,
the product contributes $1$ to the sum if and only if
the sum of the $\alpha_i$ is $n$ and $\mu_i$ divides $\alpha_i$.
Therefore this expression represents the number
of weak compositions of size $n$ and length $\ell(\mu)$
such that $\mu_i | \alpha_i$ for all $i$.
\end{proof}

\begin{example} \label{ex:evalh2}
To evaluate $h_2[\Xi_\mu]$, we can reduce
the computation by expressing it in terms of
$h_n[\Xi_r]$.
\begin{equation}
h_2[\Xi_\mu] = \sum_{i=1}^{\ell(\mu)}
h_2[\Xi_{\mu_i}] + \sum_{1 \leq i < j \leq \ell(\mu)}
h_1[\Xi_{\mu_i}] h_1[\Xi_{\mu_j}]~.
\end{equation}
Now we know from Proposition \ref{prop:evalhnXr}
that $h_2[\Xi_{\mu_i}] = 1$ if and only if $\mu_i = 1$ or $2$, hence
\begin{equation}
\sum_{i=1}^{\ell(\mu)}
h_2[\Xi_{\mu_i}] = m_1(\mu) + m_2(\mu)~.
\end{equation}
In addition, Proposition \ref{prop:evalhnXr} implies
$h_1[\Xi_{\mu_i}]=1$ if and only if $\mu_i=1$
hence
\begin{equation}
\sum_{1 \leq i < j \leq \ell(\mu)}
h_1[\Xi_{\mu_i}] h_1[\Xi_{\mu_j}] = \pchoose{m_1(\mu)}{2}~.
\end{equation}

Alternatively, we can see this evaluation in terms of
its computation of the combinatorial interpretation.
We have that this is the number of weak compositions of
length $\ell(\mu)$ of the form $(0^{i-1},2,0^{\ell(\mu)-i})$
where the 2 is in the position $i$ where $\mu_i=1$ or $2$
or of the form $(0^{i-1},1,0^{j-i-1},1,0^{\ell(\mu)-j})$
where $\mu_i=\mu_j=1$.  Clearly there are $m_1(\mu) + m_2(\mu)$
of the first type and $\pchoose{m_1(\mu)}{2}$ of the
second.
\end{example}

Equation \eqref{eq:hnatXimu} allows us to give a
combinatorial interpretation for
$h_\lambda[\Xi_\mu]$ in terms of sequences of compositions.

\begin{definition}
Define the set $\bC_{\lambda,\mu}$ to be the
sequences $\alpha^{(\ast)} = (\al^{(1)},
\al^{(2)}, \cdots, \al^{(\ell(\la))})$ of weak compositions
$\alpha^{(i)} \models_w \lambda_i$ such that
$\ell(\alpha^{(i)}) = \ell(\mu)$ and $\mu_j$
divides $\alpha^{(i)}_j$ for each $1 \leq i \leq \ell(\la)$.
\end{definition}

\begin{example} \label{ex:bC} To compute $\bC_{(3,1),(3,3,2,2,1)}$ we
count pairs $(\alpha,\beta)$
where $\alpha$ is a weak composition of $3$ of length
$5$ and $\beta$ is a weak composition of $1$ of length
$5$ such that $\mu_i$ divides $\alpha_i$ and $\beta_i$
where $\mu = (3,3,2,2,1)$.  There are $5$ ways of doing
this given by the following pairs of compositions
\begin{equation*}
((00003),(00001)), ((00021),(00001)), ((00201),(00001))
\end{equation*}
\begin{equation*}
((03000),(00001)), ((30000),(00001))
\end{equation*}
\end{example}

\begin{prop}\label{prop:heval}
For partitions $\lambda$ and $\mu$,
\begin{equation}
h_\lambda[\Xi_\mu] = |\bC_{\lambda,\mu}|~.
\end{equation}
\end{prop}

\begin{proof}
Since $h_{r}[\Xi_\mu]$
is equal to the number of weak compositions $\alpha$ of
length $\ell(\mu)$ and size $r$
such that $\mu_j | \alpha_j$ (that is, it is
equal to $|\bC_{(r),\mu}|$)
hence, by the multiplication
principle, the expression
\begin{equation}
h_\lambda[\Xi_\mu] = h_{\la_1}[\Xi_\mu]h_{\la_2}[\Xi_\mu]
\cdots h_{\la_{\ell(\la)}}[\Xi_\mu],
\end{equation}
is equal to the number of sequences of compositions
$\alpha^{(\ast)} = (\al^{(1)},
\al^{(2)}, \cdots, \al^{(\ell(\la))})$
such that $\alpha^{(i)}$ is a weak composition of $\lambda_i$
and of length $\ell(\mu)$ with $\mu_j$ divides $\alpha^{(i)}_j$
for each $1 \leq j \leq \ell(\mu)$.
\end{proof}

This combinatorial interpretation for
$h_\lambda[\Xi_\mu]$ is only the starting point.
We will give an expression for this quantity in
terms of symmetric function coefficients.

We can state this as the following combinatorial
result.
\begin{prop}\label{prop:evalht}
For partitions $\nu$ and $\mu$,
$
\HX_{\nu,\mu} :=
\left< h_\nu h_{|\mu|-|\nu|}, p_\mu\right>
$
is equal to the number of ways that
some of the cells of the diagram of $\mu$ can be filled with
the labels $\{1,2,\ldots,\ell(\nu)\}$
such that the cells in the same row all have the same label
and in total $\nu_j$ cells are labeled with
the integer $j$ for $1 \leq j \leq \ell(\nu)$.
\end{prop}

\begin{proof}
Since $\HX_{\nu, \mu}$ is equal to the coefficient
of $\frac{p_\mu}{z_\mu}$, it is also equal
to $z_\mu$ times the coefficient of $p_\mu$
in the symmetric function
$h_{\nu_1} h_{\nu_2} \cdots h_{\nu_{\ell(\nu)}}
h_{|\mu|- |\nu|}$.  Since $h_{\nu_i} = \sum_{\ga^{(i)} \vdash \nu_i}
\frac{p_{\ga^{(i)}}}{z_{\ga^{(i)}}}$ then one way
that we can express this is
\begin{equation}\label{eq:hhtsum}
\HX_{\nu, \mu} =
\sum_{\ga^{(\ast)}} \frac{z_\mu}{z_{\ga^{(1)}} z_{\ga^{(2)}}
\cdots
z_{\mu \backslash \bigcup_i \ga^{(i)}}}
\end{equation}
where the sum is
over all sequences of partitions
$\ga^{(\ast)} = (\ga^{(1)}, \ga^{(2)}, \ldots,
\ga^{(\ell(\nu))})$ with
$\ga^{(i)}$ a partition of $\nu_{i}$ for $1\leq i \leq \ell(\nu)$
and such that the parts of
$\bigcup_i \ga^{(i)} =
\ga^{(1)} \cup \ga^{(2)} \cup \cdots \cup \ga^{(\ell(\nu))}$
are a subset of the parts of $\mu$.
The last partition that appears in this denominator
is $\mu\backslash \bigcup_i \ga^{(i)}$ and this is a partition
of size $|\mu| - |\nu|$ of all the parts of
$\mu$ which are not used in
$\ga^{(1)} \cup \ga^{(2)} \cup \cdots \cup \ga^{(\ell(\nu))}$.
We are using the convention that $h_{-r}=0$ for $r>0$, hence
$\HX_{\nu, \mu} =0$ if $|\nu|>|\mu|$ since we will have
$h_{|\mu|-|\nu|}=0$.

Now recall that $z_\mu =
\prod_{i=1}^{\mu_1} i^{m_i(\mu)} m_i(\mu)!$.
This implies that
\begin{align}
\frac{z_\mu}{z_{\ga^{(1)}} z_{\ga^{(2)}}
\cdots z_{\ga^{(\ell(\nu))}} z_{\mu \backslash \bigcup_i \ga^{(i)}}}
&= \prod_{i=1}^{\mu_1} \frac{m_i(\mu)!}{m_i(\ga^{(1)})!
m_i(\ga^{(2)})! \cdots m_i(\ga^{(\ell(\nu))})!
m_i(\mu\backslash\bigcup_i \ga^{(i)})!}\nonumber\\
&= \prod_{i=1}^{\mu_1}
\pchoose{m_i(\mu)}{
m_i(\ga^{(1)}),
m_i(\ga^{(2)}), \ldots, m_i(\ga^{(\ell(\nu))})}~.\label{eq:multiptnchoose}
\end{align}

We have established by Equations \eqref{eq:hhtsum} and
\eqref{eq:multiptnchoose} that
\begin{equation}
\HX_{\nu,\mu} =
\left< h_{|\mu|-|\nu|} h_\nu, p_\mu\right> =
\sum_{\ga^{(\ast)}}
\prod_{i=1}^{\mu_1}
\pchoose{m_i(\mu)}{
m_i(\ga^{(1)}),
m_i(\ga^{(2)}), \ldots, m_i(\ga^{(\ell(\nu))})}
\end{equation}
where the sum is over all sequences of partitions
$\ga^{(\ast)} = (\ga^{(1)}, \ga^{(2)}, \ldots,
\ga^{(\ell(\nu))})$ with
$\ga^{(j)}$ a partition of
$\nu_{j}$ for $1\leq j \leq \ell(\nu)$
and such that the parts of
$\bigcup_i \ga^{(i)} =
\ga^{(1)} \cup \ga^{(2)} \cup \cdots \cup \ga^{(\ell(\nu))}$
are a subset of the parts of $\mu$.

Take a filling of some of the rows of the partition
$\mu$ such that
the all the cells in a row are assigned the same label and the total number
of cells with label $j$ is equal to $\nu_j$.
The rows of $\mu$ that are labeled with $j$
determine a partition $\ga^{(j)}$ of size $\nu_j$.
This determines the sequence $\ga^{(\ast)}$.

Now among the $m_i(\mu)$ parts of the partition $\mu$
that are of size $i$, $m_i(\ga^{(j)})$ parts are labeled
with label $j$.  There are precisely
\begin{equation}
\pchoose{m_i(\mu)}{
m_i(\ga^{(1)}),
m_i(\ga^{(2)}), \ldots, m_i(\ga^{(\ell(\nu))})}
\end{equation}
ways this can be done.  In other words, there is a
sequence of positions which determines where the
$m_i(\ga^{(j)})$ parts of size $i$
are chosen from the $m_i(\mu)$ parts of $\mu$ of size $i$.

The reverse bijection is found by taking a sequence of
partitions $\ga^{(\ast)}$ and sequence of positions
which tells us which of the $m_i(\mu)$ parts of size $i$
of the partition which are filled with the labels $j$.
\end{proof}

\begin{example}\label{ex:HX}
To calculate 
$\HX_{(4),(3,3,2,2,1,1)}$
we need to label
some of the rows of the partition with just one label
such that the number of cells labeled is equal to $4$.
The following diagrams are the only ways that this can
be done.
\begin{equation*}
\young{\cr1\cr&\cr&\cr&&\cr1&1&1\cr}
\hskip .2in
\young{\cr1\cr&\cr&\cr1&1&1\cr&&\cr}
\hskip .2in
\young{1\cr\cr&\cr&\cr&&\cr1&1&1\cr}
\hskip .2in
\young{1\cr\cr&\cr&\cr1&1&1\cr&&\cr}
\hskip .2in
\young{\cr\cr1&1\cr1&1\cr&&\cr&&\cr}
\hskip .2in
\young{1\cr1\cr&\cr1&1\cr&&\cr&&\cr}
\hskip .2in
\young{1\cr1\cr1&1\cr&\cr&&\cr&&\cr}
\end{equation*}
Therefore
$\HX_{(4),(3,3,2,2,1,1)}=7$.
Similarly, to compute
$\HX_{(3,1),(3,3,2,2,1,1)}$.
we fill the rows of diagram of the partition
$(3,3,2,2,1,1)$ with labels of three $1$'s and a $2$
such that the whole row is given the same label.
\begin{equation*}
\young{\cr2\cr&\cr&\cr&&\cr1&1&1\cr}
\hskip .2in
\young{\cr2\cr&\cr&\cr1&1&1\cr&&\cr}
\hskip .2in
\young{2\cr\cr&\cr&\cr&&\cr1&1&1\cr}
\hskip .2in
\young{2\cr\cr&\cr&\cr1&1&1\cr&&\cr}
\hskip .2in
\young{1\cr2\cr&\cr1&1\cr&&\cr&&\cr}
\hskip .2in
\young{1\cr2\cr1&1\cr&\cr&&\cr&&\cr}
\hskip .2in
\young{2\cr1\cr&\cr1&1\cr&&\cr&&\cr}
\hskip .2in
\young{2\cr1\cr1&1\cr&\cr&&\cr&&\cr}
\end{equation*}
Therefore
$\HX_{(3,1),(3,3,2,2,1,1)}=8$.
\end{example}

Recall that we indicate that $\pi$ is a multiset partition
of the multiset with the letter $i$ occurring $\la_i$ times
($1 \leq i \leq \ell(\la)$) by
the symbol $\pi \mvdash
\dcl1^{\la_1},2^{\la_2},\ldots,\ell^{\la_\ell}\dcr$.
Recall also that $\mt(\pi)$ is the integer
partition of size equal to the number of parts in $\pi$
where the parts of $\mt(\pi)$ are the multiplicities
of the multisets which appear in $\pi$.

A multiset partition of a multiset does not naturally have order
on the parts but we will need one for establishing
a unique mapping.
The order is not especially important,
but we need it to be consistent with how we label column strict tableaux
so we choose to put the sets which occur most often first,
and among those that occur the same number of times
we use lexicographic order if the elements of the set
are read in increasing order.

\begin{example}
In the multiset $\dcl1^{12},2^7,3^2\dcr$
we have an example multiset partition
\begin{equation}
\pi =
\dcl\dcl1,2\dcr,\dcl1,2\dcr,\dcl1,2\dcr,\dcl1,1,1\dcr,\dcl1,1,1\dcr,
\dcl1,2,2,3\dcr,\dcl1,2,2,3\dcr,\dcl1\dcr\dcr
\end{equation}
The parts $\dcl1,2\dcr$ occur first because they occur 3 times,
$\dcl1,1,1\dcr$ and $\dcl1,2,2,3\dcr$ each occur twice so they are each
second and $\dcl1,1,1\dcr$ is before $\dcl1,2,2,3\dcr$ because
$111 <_{lex} 1223$.  Finally $\dcl1\dcr$ only occurs once, hence
it is last.
\end{example}

\begin{definition}
Let $\bP_{\lambda,\mu}$ be the set of pairs
$(\pi, T)$ where $\pi$ is a multiset partition of the multiset
$\dcl1^{\la_1},2^{\la_2},\ldots,\ell^{\la_\ell}\dcr$
and $T$ is a filling of some of the cells
of the diagram of the partition $\mu$
with content $\ga = \mt(\pi)$ and all cells
in the same row have the same label.
\end{definition}

A restatement of Proposition \ref{prop:evalht} would be that
\begin{equation}
\sum_{\pi \mvdash \dcl1^{\la_1},2^{\la_2},\ldots,\ell^{\la_\ell}\dcr}
\HX_{\mt(\pi),\mu} = |\bP_{\lambda,\mu}|~.
\end{equation}

\begin{example}\label{ex:bP}
The set $\bP_{(3,1),(3,3,2,2,1)}$
consists of the following $5$ pairs of
multiset partitions
of the multiset $\dcl1,1,1,2\dcr$ and fillings of the
diagram for $(3,3,2,2,1)$.
\begin{equation*}
(\dcl\dcl1,1,1,2\dcr\dcr, \young{1\cr&\cr&\cr&&\cr&&\cr}),\hskip .1in
(\dcl\dcl1\dcr,\dcl1\dcr,\dcl1,2\dcr\dcr,
\young{2\cr1&1\cr&\cr&&\cr&&\cr}),\hskip .1in
(\dcl\dcl1\dcr,\dcl1\dcr,\dcl1,2\dcr\dcr, \young{2\cr&\cr1&1\cr&&\cr&&\cr}),
\end{equation*}
\begin{equation*}
(\dcl\dcl1\dcr,\dcl1\dcr,\dcl1\dcr,\dcl2\dcr\dcr,
\young{2\cr&\cr&\cr1&1&1\cr&&\cr}),\hskip .1in
(\dcl\dcl1\dcr,\dcl1\dcr,\dcl1\dcr,\dcl2\dcr\dcr, \young{2\cr&\cr&\cr&&\cr1&1&1\cr})
\end{equation*}
\end{example}

\begin{definition}
Let $\bT_{\lambda,\mu}$ be the set of fillings of
some of the cells of the partition $\mu$ with
content $\lambda$ such that any number of
labels can go into the
same cell but all cells in the same row must have
the same multiset of labels.
\end{definition}

\begin{example}\label{ex:bT}
The set $\bT_{(3,1),(3,3,2,2,1)}$ consists of the
following $5$ fillings
\begin{equation*}
\young{w\cr&\cr&\cr&&\cr&&\cr}
\hskip .2in
\young{12\cr1&1\cr&\cr&&\cr&&\cr}
\hskip .2in
\young{12\cr&\cr1&1\cr&&\cr&&\cr}
\hskip .2in
\young{2\cr&\cr&\cr1&1&1\cr&&\cr}
\hskip .2in
\young{2\cr&\cr&\cr&&\cr1&1&1\cr}
\end{equation*}
where the $w$ in the first diagram is $w = 1112$
($w$ represents the multiset of labels $\dcl1,1,1,2\dcr$).
\end{example}

The last two examples suggest the following
proposition.

\begin{prop}\label{prop:rhseval}
For partitions $\lambda$ and $\mu$,
\begin{equation}
\sum_{\pi \mvdash \dcl1^{\la_1},2^{\la_2},\ldots,\ell^{\la_\ell}\dcr}
\HX_{\mt(\pi),\mu}
=|\bT_{\lambda,\mu}|~.
\end{equation}
\end{prop}

\begin{proof}
By Proposition \ref{prop:evalht} we have established that
\begin{equation}
\sum_{\pi \mvdash \dcl1^{\la_1},2^{\la_2},\ldots,\ell^{\la_\ell}\dcr}
\HX_{\mt(\pi),\mu} = |\bP_{\lambda,\mu}|,
\end{equation}
hence it remains to show that there is a bijection
between the elements of $\bP_{\lambda,\mu}$ and
$\bT_{\lambda,\mu}$.

Take a pair $(\pi, T) \in \bP_{\lambda,\mu}$.  Since the
order on the parts of the multiset partition of $\pi$ puts
them in weakly decreasing order dependent on the number of times that they occur, the label $1$ in $T$ can be replaced by the
first multiset which occurs in $\pi$, the $2$ can be
replaced by the second multiset which occurs, etc.
The result will be a filling $T'$ which is an element
of $\bT_{\lambda,\mu}$ because all rows of $T$
have the same labels (and so will be the case with $T'$)
and the content of $T'$ will be the same as the
multiset and so will be $\dcl1^{\la_1},2^{\la_2},
\ldots, \ell^{\la_\ell}\dcr$.

\begin{example} \label{ex:bigfilling}
Consider the pair $(\pi,T) \in \bP_{(12,7,2),(3,3,2,2,1)}$
where
$$\pi =
\dcl\dcl1,2\dcr,\dcl1,2\dcr,\dcl1,2\dcr,\dcl1,1,1\dcr,\dcl1,1,1\dcr,
\dcl1,2,2,3\dcr,\dcl1,2,2,3\dcr,\dcl1\dcr\dcr
$$ and
$T$ is the filling
\begin{equation*}
\young{4\cr
2&2\cr
3&3\cr
&&\cr
1&1&1\cr}
\end{equation*}
This is mapped to the filling $T'$ equal to
\begin{equation*}
\Young{1\cr
\hbox{\tiny 111}&\hbox{\tiny 111}\cr
\hbox{\Tiny 1223}&\hbox{\Tiny 1223}\cr
&&\cr
12&12&12\cr}
\end{equation*}
where $T' \in \bT_{(12,7,2),(3,3,2,2,1)}$.
The map is to replace the labels of $T$ with the parts
of $\pi$.
\end{example}

As long as the order in which the multisets occur in
$\pi$ is fixed, the map from pairs $(\pi, T) \in \bP_{\lambda,\mu}$
to $T' \in \bT_{\lambda,\mu}$ is invertible since we
can recover the multiset partition of a multiset
$\pi$ as the set of labels in $T'$ and the filling
$T$ is just the filling $T'$ with each of the sets
replaced by the integer order in which the set appears
in $\pi$.
\end{proof}

We are now prepared to provide a proof of the following
result which was first stated as Theorem \ref{thm:hrts}
in Section \ref{sec:symgrpchar} and used to provide
a definition of the induced trivial character basis.

\begin{theorem}
For all partitions $\lambda$,
\begin{equation}
h_\lambda[\Xi_\mu] =
\sum_{\pi \mvdash \dcl1^{\la_1},2^{\la_2},\ldots,\ell^{\la_\ell}\dcr}
\HX_{\mt(\pi),\mu}~.
\end{equation}
\end{theorem}

\begin{proof}
We have already established in Proposition
\ref{prop:heval} that
$h_\lambda[\Xi_\mu] = |\bC_{\lambda,\mu}|$
and in Proposition \ref{prop:rhseval}
that
\begin{equation}
\sum_{\pi \mvdash \dcl1^{\la_1},2^{\la_2},\ldots,\ell^{\la_\ell}\dcr}
\HX_{\mt(\pi),\mu}
= |\bP_{\lambda,\mu}|
= |\bT_{\lambda,\mu}|~.
\end{equation}
In order to show this result
we will provide a bijection between $\bC_{\lambda,\mu}$
and $\bT_{\lambda,\mu}$.

We start with a filling $T \in \bT_{\lambda,\mu}$ and
define a list of weak compositions
$$\al^{(\ast)} =
(\al^{(1)}, \al^{(2)}, \ldots, \al^{(\ell(\la))})$$
where $\al^{(d)}_i$ is equal to the number of labels $d$
in the $i^{th}$ row of $T$.  Since the content of
$T$ is equal to the multiset $\lambda$, we know that
$\al^{(d)}$ will be a weak composition of $\lambda_d$
and because in row $i$ the multiset of labels is the same
for each cell in a row $\mu_i$, it must be that $\mu_i$
divides $\al^{(d)}_i$.

This procedure is reversible since starting with a sequence
of weak compositions $\alpha^{(\ast)} \in \bC_{\lambda,\mu}$,
we can place $\al^{(d)}_i/\mu_i$ labels of $d$ in each cell
of the $i^{th}$ row of the diagram $\mu$ to recover the
filling $T' \in \bT_{\lambda,\mu}$.

We conclude that
\begin{equation}
h_\lambda[\Xi_\mu] = |\bC_{\lambda,\mu}|
= |\bT_{\lambda,\mu}|
= |\bP_{\lambda,\mu}|
=\sum_{\pi \mvdash \dcl1^{\la_1},2^{\la_2},\ldots,\ell^{\la_\ell}\dcr}
\HX_{\mt(\pi),\mu}~.\qedhere
\end{equation}
\end{proof}
\end{section}

\begin{section}{Appendix II: Roots of unity and zeros of
multivariate polynomials} \label{sec:rootsofunity}

The goal of this section is to prove the following
basic proposition.
\begin{prop}\label{prop:rootsimplsf} Let $f,g \in Sym$ be symmetric functions
of degree less than or equal to some positive integer $n$.
Assume that $$f[\Xi_\ga] = g[\Xi_\ga]$$ for all partitions
$\ga$ such that $|\ga| \leq n$, then
$$f=g$$
as elements of $Sym$.
\end{prop}

Our plan for the proof is to reduce this proposition 
to the fact that a univariate polynomial of degree
$n$ has at most $n$ zeros and hence a univariate polynomial
of degree at most $n$ with more than $n$ zeros must be
the $0$ polynomial.

To begin, notice that
$p_k[\Xi_r] = r$ if $r$ divides $k$ and it is equal to $0$
otherwise.
In general, we can express any partition
$\gamma$ in exponential notation
$\gamma = (1^{m_1} 2^{m_2}\cdots r^{m_r})$
where $m_i$ is the number of parts of size $i$ in $\ga$.
Recall that $p_k[\Xi_\gamma] = p_k[\Xi_{\gamma_1}] + p_k[\Xi_{\gamma_2}] +
\cdots + p_k[\Xi_{\gamma_{\ell(\gamma)}}]$, so
\begin{equation}
p_k[\Xi_\gamma] = \sum_{d|k} d m_d~.
\end{equation}
Hence any symmetric function $f$ evaluated at
some set of roots of unity is equal to
a polynomial in values $m_1, m_2, \ldots, m_n$
where
\begin{equation} \label{eq:sftopoly}
f[\Xi_\gamma] = f \Big|_{p_k \rightarrow \sum_{d|k} d m_d}
= q( m_1, m_2, \ldots, m_n )~.
\end{equation}
Moreover, if we know this polynomial $q(m_1, m_2, \ldots, m_d)$
we can use M\"obius inversion to recover
the symmetric function since if $p_k = \sum_{d|k} d m_d$,
then $k m_k = \sum_{d|k} \mu(k/d) p_d$ where
\begin{equation}
\mu(r) = \begin{cases}(-1)^d&\hbox{if $r$ is a
product of $d$ distinct primes}\\
0&\hbox{if $r$ is not square free}
\end{cases}.
\end{equation}
Therefore, we also have
\begin{equation}\label{eq:recoversf}
q\!\!\left( p_1, \frac{p_2-p_1}{2},\frac{p_3-p_1}{3},
\ldots, \frac{1}{n}\sum_{d|n} \mu(n/d) p_d \right) = f~.
\end{equation}

To show that
Proposition \ref{prop:rootsimplsf} is true, we will
prove that if $h[\Xi_\gamma] = 0$ for all partitions $|\gamma|\leq n$,
then $h = 0$ as a symmetric function where $h = f-g$.
We will do this by considering $h$ as $q(x_1, x_2, \ldots, x_n)$
where $x_r$ is replaced by $\frac{1}{r}\sum_{d|r}\mu(r/d) p_d$
and that this multivariate polynomial
evaluates to $0$ for all
$(m_1, m_2, \ldots, m_n)$
where $\gamma= (1^{m_1}2^{m_2}\cdots n^{m_n})$
is a partition with $|\gamma|\leq n$.
This is a consequence of the next lemma below.

Define the degree of a monomial so that $deg(x_i)=i$
and hence $deg(x_1^{a_1}x_2^{a_2} \cdots x_k^{a_k}) =
a_1 + 2a_2 + 3a_3+ \cdots + k a_k$.

\begin{lemma}\label{lem:polyzeros}
Let $q(x_1, x_2, \ldots, x_n)$ be an element in a multivariate
polynomial ring $\QQ[x_1, x_2,$  $\ldots, x_n]$ with
$deg(q(x_1, x_2, \ldots, x_n))\leq d$ for some $d$.
If $q(m_1, m_2, \ldots, m_n) = 0$ for all sequences
$(m_1, m_2, \ldots, m_n)$ with $m_i \geq 0$ and
$m_1 + 2m_2 + 3m_3+ \cdots +n m_n \leq d$, then
$q(x_1, x_2, \ldots, x_n) = 0$.
\end{lemma}

\begin{proof}  We argue by induction on the number of variables $n$.
First we note that if $n=1$, then
if $q(x_1)$ is a polynomial of degree
$\leq d$ and $q(0) = q(1) = \cdots = q(d)=0$, then $q(x_1)=0$
because we know that a univariate polynomial of degree $r$
can have at most $r$ roots.

Let our induction assumption be that,
if $q(m_1, m_2, \ldots, m_{n-1}) = 0$ for all sequences
$(m_1, m_2, \ldots, m_{n-1})$ with $m_i \geq 0$ and
$m_1 + 2m_2 + 3m_3+ \cdots +(n-1) m_{n-1} \leq d$, then
$q(x_1, x_2, \ldots, x_{n-1}) = 0$.

Now assume that our inductive hypothesis is true and consider
a polynomial in $n$ variables,
\begin{equation}
q(x_1, x_2, \ldots, x_n) =
\sum_{i=0}^r
q^{(i)}(x_1, x_2, \ldots, x_{n-1})x_n^i
\end{equation}
where $r$ is the maximum power of the variable $x_n$ in the polynomial $q$
and $0 \leq r \leq d/n$ and the coefficient of $x_n^i$ is
$q^{(i)}(x_1, x_2, \ldots, x_{n-1})$ a multivariate polynomial of
degree less than or equal to $d-ni$.

We wish to show that $q(x_1, x_2, \ldots, x_n)$ is in fact $0$.  Assume 
 $r$ is the largest
exponent of $x_n$ for which
there is a non-zero coefficient,
then fix $(m_1, m_2, \ldots, m_{n-1})$ such that
$m_1 + 2m_2 + \cdots + (n-1)m_{n-1} \leq d-rn$.  Now
\begin{equation}
q(m_1, m_2, \ldots, m_{n-1}, m_n) = 0
\end{equation}
for each $m_n = 0, 1, 2, \ldots, r$, hence
$q(m_1, m_2, \ldots, m_{n-1}, x_n) = 0$
because it is a polynomial of degree at most $r$ with
more than $r$ roots and in particular
the coefficient
$q^{(r)}(m_1, m_2, \ldots, m_{n-1})$
is equal to $0$.
But now we have that
$q^{(r)}(x_1, x_2, \ldots, x_{n-1})$
is a polynomial of degree $\leq d-rn$
in $n-1$ variables which vanishes
at all $(x_1, x_2, \ldots, x_{n-1})=
(m_1, m_2, \ldots, m_{n-1})$ with
$m_1 + 2m_2 + \cdots + (n-1)m_{n-1} \leq d-rn$
and hence $q^{(r)}(x_1, x_2, \ldots,$ $x_{n-1})=0$
by our induction hypothesis.
\end{proof}

Hence, Proposition \ref{prop:rootsimplsf}
follows as a corollary.
\begin{proof}{(of Proposition \ref{prop:rootsimplsf})}
We will show the equivalent statement, by setting $h = f-g$,
that if $h$ is a symmetric function
of degree less than or equal to $n$ and $h[\Xi_\gamma] = 0$ for all
$|\gamma|\leq n$, then $h=0$.

This statement is directly equivalent to Lemma \ref{lem:polyzeros}
since if we replace $p_k$ in $h$ with $\sum_{d|k} d x_d$ then
$$q(x_1, x_2, \ldots, x_n) = h\Big|_{p_k \rightarrow \sum_{d|k} d x_d}$$
is a polynomial of degree at most $n$ that has the property that
$q(m_1, m_2, \ldots, m_n) = h[\Xi_\gamma] = 0$ for all
$\gamma = (1^{m_1}2^{m_2}\cdots n^{m_n})$ for all $|\gamma|\leq n$.
By Lemma \ref{lem:polyzeros}, $q(x_1, x_2, \ldots, x_n)=0$
and hence by equation \eqref{eq:recoversf}, $h = 0$.
\end{proof}

We considered the case in Proposition \ref{prop:rootsimplsf} that
$f[\Xi_\gamma] = g[\Xi_\gamma]$ implies $f=g$ when $\gamma$ is small, but
more often we will know that $f[\Xi_\gamma] = g[\Xi_\gamma]$ for all
$\gamma$ which are partitions of integers greater than or equal to some value $n$.
We can reduce the implication to the previous case in the following
proposition.

\begin{cor}\label{cor:evalsf} Let $f,g \in Sym$ be symmetric functions
of degree less than or equal to some positive integer $n$.
Assume that $$f[\Xi_\gamma] = g[\Xi_\gamma]$$ for all partitions
$\gamma$ such that $|\gamma| \geq n$, then
$$f=g$$
as elements of $Sym$.
\end{cor}

\begin{proof}
We reduce the conditions of this corollary to the previous
case by considering partitions $\gab$ of size less than or equal
to $n$ and then let $\ga = (n+1,\gab)$.  Since for all $k \leq n$,
$p_k[\Xi_{n+1}] = 0$, then $p_k[\Xi_\ga] = p_k[\Xi_{\gab}]$.

Since $f$ (and similarly $g$) are of degree less than or equal to
$n$, then
$f = \sum_{|\lambda|\leq n} c_\lambda p_\lambda$ for some coefficients
$c_\lambda$ then $p_\lambda[\Xi_\ga] = p_\lambda[\Xi_\gab]$ since
$p_\lambda[\Xi_{n+1}]=0$.
Therefore $f[\Xi_\ga]=f[\Xi_\gab]$ and $g[\Xi_\ga] =
g[\Xi_\gab]$ and hence $f[\Xi_\gab] = g[\Xi_\gab]$.  By
Proposition \ref{prop:rootsimplsf} we can conclude that $f=g$.
\end{proof}
\end{section}

\begin{section}{Appendix III: Using Sage to compute
the character bases of symmetric functions}

This section does not appear in the journal published version
of this paper \cite{OZa} and is an expanded version of
a tutorial written in an extended abstract \cite{OZb} that
the authors uploaded to the arXiv so that code could be
reviewed for Sage \cite{sage, sage-co}.  The purpose of this
section is to give examples of the use of the irreducible character
basis and the induced trivial character basis as a computation tool in Sage.

\begin{subsection}{Symmetric functions in Sagemath}
Sage is an open source symbolic calculation program based on
the computer language Python.  It may be downloaded from the website:
\begin{center}
\url{https://www.sagemath.org/}
\end{center}
Computations may be done online without downloading the program at
\begin{center}
\url{https://sagecell.sagemath.org/} or at \href{https://cocalc.com/}{CoCalc}~.
\end{center}

The following two online tutorials will be useful to cover the basic functionality
of symmetric functions in Sage:
\begin{itemize}
\item \href{https://more-sagemath-tutorials.readthedocs.io/en/latest/tutorial-symmetric-functions.html}{Symmetric Functions Tutorial}
\item \href{https://doc.sagemath.org/html/en/reference/combinat/sage/combinat/sf/sf.html}{Documentation: Symmetric functions, with their multiple realizations}
\end{itemize}
This appendix will focus on the use of the character bases which are only superficially covered
in other documentation and tutorials.

A large community of mathematicians
participate in its support and add to its functionality
\cite{sage, sage-co}.  In particular, the built-in library
for symmetric functions includes a large extensible set of functions
which makes it possible to do calculations within the ring
following closely the mathematical notation that we use in this paper.
The language itself has a learning curve, but the contributions
made by the community towards the functionality
make that a barrier worth overcoming.

In version 6.10 or later of Sage (released January 2016)
these bases are available
as methods in the ring of symmetric functions.
At the time of this writing the current version is 9.3
(released May 2021).
\end{subsection}

\begin{subsection}{Defining bases}
We demonstrate examples of some of the definitions and results in this paper
by Sage calculations.  The first step to using the symmetric function code
is to define bases of the ring over the field of rational numbers.
\begin{verbatim}
sage: Sym = SymmetricFunctions(QQ)
sage: Sym
Symmetric Functions over Rational Field
sage: st = Sym.irreducible_symmetric_group_character(); st
Symmetric Functions over Rational Field in the irreducible character basis
sage: ht = Sym.induced_trivial_character(); ht
Symmetric Functions over Rational Field in the induced trivial symmetric
group character basis
\end{verbatim}

Instead of defining each basis one at a time,
alternatively there is a command to define all of the bases of the symmetric functions
that do not use a parameter with a single command:
\begin{verbatim}
sage: SymmetricFunctions(QQ).inject_shorthands('all', verbose=False)
\end{verbatim}
Note that by setting {\tt verbose=True} in that command will list
all of the bases defined.  For the commands in the examples below we will assume
that all of these bases have been added to the namespace.  A table with a list
of the bases defined by this command (in Sage version 9.3)
appears in Figure \ref{fig:namespace}.

\begin{figure}[h]\label{fig:namespace}
\begin{center}
\begin{tabular}{c|c}
mathematical name&Sage shorthand\cr
\hline
Schur&{\tt s}\cr
complete/homogeneous&{\tt h}\cr
elementary&{\tt e}\cr
monomial&{\tt m}\cr
power&{\tt p}\cr
forgotten&{\tt f}\cr
\end{tabular}
\begin{tabular}{c|c}
mathematical name&Sage shorthand\cr
\hline
irreducible character&{\tt st}\cr
induced trivial character&{\tt ht}\cr
orthogonal&{\tt o}\cr
symplectic&{\tt sp}\cr
Witt&{\tt w}\cr\cr
\end{tabular}
\end{center}
\caption{A list of bases defined by the {\tt inject\_shorthands('all')} command.}
\end{figure}
\end{subsection}

\begin{subsection}{Structure coefficients}
Theorem \ref{th:uniqueness} part (3) states that the structure
coefficients of the irreducible character basis are equal to the
reduced Kronecker coefficients.
We can compare the structure coefficients of the $\st$-basis with
the Kronecker product of Schur functions whose first part is
sufficiently large.  If the first row of the indexing partition
is removed from the terms in the expression for the Kronecker
product, there is a corresponding term in the product
of the irreducible character basis.
\begin{verbatim}
sage: st[2]*st[2]
st[] + st[1] + st[1, 1] + st[1, 1, 1] + 2*st[2] + 2*st[2, 1] + st[2, 2]
+ st[3] + st[3, 1] + st[4]
sage: s[6,2].kronecker_product(s[6,2])
s[4, 2, 2] + s[4, 3, 1] + s[4, 4] + s[5, 1, 1, 1] + 2*s[5, 2, 1] + s[5,
3] + s[6, 1, 1] + 2*s[6, 2] + s[7, 1] + s[8]
\end{verbatim}
\end{subsection}

\begin{subsection}{Change of bases}
If {\tt f} is an element of the ring of symmetric functions
and {\tt basis} is a basis of the symmetric functions then {\tt basis(f)} expresses the
element in {\tt f} in that basis.  For example, the symmetric function
in Example \ref{ex:htost} and Example \ref{ex:settableaux} are computed with the commands:
\begin{verbatim}
sage: st(h[2,1]) # express h_21 in the st-basis
4*st[] + 7*st[1] + 3*st[1, 1] + 4*st[2] + st[2, 1] + st[3]
sage: st(h[1,1,1,1])                                                                                                                                           
15*st[] + 37*st[1] + 31*st[1, 1] + 10*st[1, 1, 1] + st[1, 1, 1, 1] +
31*st[2] + 20*st[2, 1] + 3*st[2, 1, 1] + 2*st[2, 2] + 10*st[3] +
3*st[3, 1] + st[4]
\end{verbatim}

The main result of \cite{AssafSpeyer} is that the Schur expansion of
the irreducible character basis is alternating in sign by degree.
This can be observed in the following example calculation:
\begin{verbatim}
sage: s(st[2,2])
-s[1] + 4*s[1, 1] + 2*s[2] - 2*s[2, 1] + s[2, 2] - s[3]
\end{verbatim}
\end{subsection}

\begin{subsection}{Eigenvalues of a permutation matrix}
In Section \ref{subsec:ringsf} we introduce the operation of evaluating
a symmetric function at the eigenvalues of a permutation matrix.  The
character bases are defined so that $\st_\lambda[\Xi_\mu]$ is equal
to the irreducible $S_{|\mu|}$ character indexed by the partition
$(|\mu|-|\lambda|, \lambda)$ at a permutation of cycle type $\mu$.
In Sage, elements of the ring of symmetric functions
have the method {\tt eval\_at\_permutation\_roots}
which represents this operation.

For instance, an example of Proposition \ref{prop:evalhnXr} with $n=8$
and for $1 \leq r \leq 10$, we compute the following list of values:
\begin{verbatim}
sage: [h[8].eval_at_permutation_roots([r]) for r in range(1,11)]                                                                                               
[1, 1, 0, 1, 0, 0, 0, 1, 0, 0]
sage: [p[8].eval_at_permutation_roots([r]) for r in range(1,11)]                                                                                               
[1, 2, 0, 4, 0, 0, 0, 8, 0, 0]
sage: [e[8].eval_at_permutation_roots([r]) for r in range(1,11)]                                                                                               
[0, 0, 0, 0, 0, 0, 0, -1, 0, 0]
\end{verbatim}

In particular the character bases have the property that they evaluate
to certain values of characters of symmetric group modules.
\begin{verbatim}
sage: ht[3,1].eval_at_permutation_roots([3,3,2,2,1])
2
sage: st[3,1].eval_at_permutation_roots([3,3,2,2,1])
-1
\end{verbatim}
In Sage, these can be compared to the following coefficients
of the power sum basis in the complete and Schur elements as follows:
\begin{verbatim}
sage: h[7,3,1].scalar(p[3,3,2,2,1])
2
sage: s[7,3,1].scalar(p[3,3,2,2,1])
-1
\end{verbatim}
\end{subsection}

\begin{subsection}{A Frobenius character map}
In the paper \cite[Equation (7)]{OZ2} we define the characteristic map
\cite[Section 4.7]{Sagan}, \cite[Section I.7, p. 113]{Mac}
(sometimes referred to as the `Frobenius map')
by the function which interprets an element $f \in Sym$ as
a symmetric group character and sends it
to a symmetric function of degree $n$ which is
a generating function for the character values of $S_n$,
\begin{equation}
\phi_n(f) = \sum_{\mu \vdash n} f[\Xi_\mu] \frac{p_\mu}{z_\mu}~.
\end{equation}
In Sage, symmetric function elements have the 
method {\tt character\_to\_frobenius\_image} which
calculates the map $\phi_n$ on the element.

The character map is the origin of our definitions since
in the beginning of our investigations
the $\hht$ and $\st$ basis for us were the pre-images of
the Schur and complete symmetric functions in the $\phi_n$ map
for $n$ sufficiently large.

We have defined the $\st$ and $\hht$ bases so that they have the property
$\phi_n(\st_\lambda) = s_{(n-|\la|,\la)}$ and $\phi_n(\hht_\lambda) = h_{(n-|\la|,\la)}$ respectively
if $n \geq |\la|+\la_1$.  If $n<|\la|+\la_1$, then the
corresponding symmetric function will still be equivalent to 
a Schur function or complete symmetric function
indexed by a list of integers.

\begin{verbatim}
sage: s(st[3,2].character_to_frobenius_image(8))
s[3, 3, 2]
sage: h(ht[3,2].character_to_frobenius_image(6))
h[3, 2, 1]
sage: s[2,1].character_to_frobenius_image(6)
s[3, 2, 1] + 2*s[4, 1, 1] + 2*s[4, 2] + 3*s[5, 1] + s[6]
\end{verbatim}
Since the Schur function is the character of an irreducible $Gl_n$ module,
the last calculation indicates that
$${\rm Res}^{Gl_6}_{S_6} V^{(32)}_{Gl_6} \simeq
V^{(321)}_{S_6} \oplus \left(V^{(411)}_{S_6}\right)^{\oplus 2} \oplus
\left(V^{(42)}_{S_6}\right)^{\oplus 2} \oplus
\left(V^{(51)}_{S_6}\right)^{\oplus 3} \oplus
V^{(6)}_{S_6}~.$$

This calculation can be compared with the expansion
of the Schur function $s_{21}$ in
the irreducible character basis:
\begin{verbatim}
sage: st(s[2,1])
st[] + 3*st[1] + 2*st[1, 1] + 2*st[2] + st[2, 1]
\end{verbatim}
\end{subsection}

\begin{subsection}{Plethysm}
Since Sage also can be used to compute plethysms
a single coefficient in this expansion can be
(inefficiently) calculated using Littlewood's formula
$$r_{\la\mu} = \left< s_\la, s_{(n-|\mu|,\mu)}[1+s_1+s_2+\cdots] \right>$$
for the restriction coefficient:
\begin{verbatim}
sage: s[5,1](1+s[1]+s[2]+s[3]).scalar(s[2,1])
3
\end{verbatim}
\end{subsection}
\vskip .2in

\begin{subsection}{Orthogonal group character basis}
As mentioned in the introduction, two additional bases of the
symmetric functions were defined by Koike and Terada \cite{KT}
which correspond to the irreducible characters of the
orthogonal and symplectic groups coming from the Weyl
character formula.
In Sage, these bases are implemented
and the orthogonal basis has the shorthand {\tt o}
while the symplectic basis has the shorthand {\tt sp}.

Since we have the containment of the orthogonal group in the
general linear group and the symmetric group (as permutation matrices)
in the orthogonal group, that is,
$$S_n \subseteq O_n \subseteq Gl_n,$$
it follows that the Schur functions (characters of irreducible $Gl_n$
modules) will have non-negative coefficients when
expressed in the orthogonal basis
and the orthogonal basis will have non-negative coefficients
when expressed in the irreducible character basis.

For instance, we compute the examples:
\begin{verbatim}
sage: o(s[2,2])                                                                                                                                                
o[] + o[2] + o[2, 2]
sage: st(o[2,2])                                                                                                                                               
st[1] + st[1, 1] + 3*st[2] + 2*st[2, 1] + st[2, 2] + st[3]
\end{verbatim}
This computation implies that for $n$ sufficiently large
$${\rm Res}^{Gl_n}_{O_n} V^{(22)}_{Gl_n} \simeq V^{(22)}_{O_n} \oplus V^{(2)}_{O_n} \oplus V^{()}_{O_n}$$
and
$${\rm Res}^{O_n}_{S_n} V^{(22)}_{O_n} \simeq
V^{(n-4,22)}_{S_n} \oplus V^{(n-3,3)}_{S_n} \oplus (V^{(n-3,21)}_{S_n})^{\oplus 2} \oplus
(V^{(n-2,2)}_{S_n})^{\oplus 3} \oplus V^{(n-2,11)}_{S_n} \oplus V^{(n-1,1)}_{S_n}~.$$
\end{subsection}
\end{section}

\end{document}